\documentclass[10pt]{amsart}
\usepackage{amsfonts}
\usepackage{amsmath}
\usepackage{amsthm}
\usepackage{amssymb}
\usepackage{latexsym}

\newtheorem{theorem}{Theorem}[section]
\newtheorem{proposition}[theorem]{Proposition}
\newtheorem{corollary}[theorem]{Corollary}
\newtheorem{lemma}[theorem]{Lemma}
\newtheorem{definition}[theorem]{Definition}
\newtheorem{remark}[theorem]{Remark}

\newcommand{\eps}{\epsilon}
\newcommand{\ga}{\gamma}
\newcommand{\Ga}{\Gamma}
\newcommand{\ka}{\kappa}
\newcommand{\la}{\lambda}
\newcommand{\La}{\Lambda}

\newcommand{\mba}{\mathbf{a}}
\newcommand{\mbA}{\mathbf{A}}
\newcommand{\mbb}{\mathbf{b}}

\newcommand{\mbc}{\mathbf{c}}

\newcommand{\mbD}{\mathbf{D}}
\newcommand{\mbe}{\mathbf{e}}

\newcommand{\mbH}{\mathbf{H}}

\newcommand{\mbm}{\mathbf{m}}

\newcommand{\mcA}{\mathcal{A}}

\newcommand{\mcD}{\mathcal{D}}

\newcommand{\mcL}{\mathcal{L}}
\newcommand{\mcO}{\mathcal{O}}

\newcommand{\mfb}{\mathfrak{b}}

\newcommand{\mfd}{\mathfrak{d}}

\newcommand{\mfg}{\mathfrak{g}}
\newcommand{\mfgl}{\mathfrak{g}\mathfrak{l}}
\newcommand{\mfh}{\mathfrak{h}}
\newcommand{\mfH}{\mathfrak{H}}

\newcommand{\mfk}{\mathfrak{k}}

\newcommand{\mfn}{\mathfrak{n}}

\newcommand{\mfp}{\mathfrak{p}}

\newcommand{\mfq}{\mathfrak{q}}

\newcommand{\mfsl}{\mathfrak{s}\mathfrak{l}}
\newcommand{\mft}{\mathfrak{t}}
\newcommand{\mfU}{\mathfrak{U}}

\newcommand{\msc}{\mathsf{c}}

\newcommand{\msd}{\mathsf{d}}
\newcommand{\msD}{\mathsf{D}}

\newcommand{\msH}{\mathsf{H}}

\newcommand{\msL}{\mathsf{L}}

\newcommand{\demo}{\noindent {\it \small Proof:}\quad}

\newcommand{\iso}{\stackrel{\sim}{\longrightarrow}}

\newcommand{\wh}{\widehat}
\newcommand{\wt}{\widetilde}
\newcommand{\lra}{\longrightarrow}

\newcommand{\into}{\hookrightarrow}
\newcommand{\onto}{\twoheadrightarrow}

\newcommand{\C}{\mathbb{C}}
\newcommand{\Z}{\mathbb{Z}}

\newcommand{\ot}{\otimes}
\newcommand{\ol}{\overline}

\newcommand{\lan}{\langle}
\newcommand{\ran}{\rangle}

\newcommand{\nwmod}[1]{W_{#1}}
\newcommand{\nwmods}[1]{W^{#1}}
\newcommand{\dhc}[2]{H_{#1,#2}}

\title{Double affine Lie algebras and finite groups}

\author{Nicolas Guay} 

\author[David Hernandez]{David Hernandez$^1$}\thanks{$^1$Supported partially by ANR through Project "G\'eom\'etrie et Structures
  Alg\'ebriques Quantiques"}

\author{Sergey Loktev}

\date{}

\begin{document}

\begin{abstract}
We begin to study the Lie theoretical analogs of symplectic
reflection algebras for $\Ga$ a finite cyclic group, which we call ``cyclic double
affine Lie algebra''. We focus on
type $A$ : in the finite (resp. affine, double affine) case, we prove
that these structures are finite (resp. affine, toroidal) type Lie
algebras, but the gradings differ. The case which is essentially new is
$\mfsl_n(\C[u,v]\rtimes\Ga)$.  We describe its universal
central extensions and start the study of its representation theory,
in particular of its highest weight integrable modules and Weyl modules. We also consider the first Weyl algebra $A_1$ instead of the polynomial ring $\C[u,v]$, and, more generally, a rank one rational Cherednik algebra.  We study quasi-finite highest weight representations of these Lie algebras.
\end{abstract}

\maketitle

\tableofcontents

\section{Introduction}

Double affine Hecke algebras have been well studied for more than fifteen
years now, although they are still very mysterious, and symplectic
reflection algebras appeared over seven years ago \cite{EtGi} as
generalizations of double affine Hecke algebras of rational type. Even
more mysterious are the double affine Lie algebras and their quantized
version introduced in \cite{gkv}, studied for instance in \cite{He1, 
He2, nag, npc, Sc2, VaVa1, VaVa2}
and the references in the survey \cite{He3}.

In this paper, we study candidates for Lie theoretical
analogs of symplectic
reflection algebras, which we call ``cyclic double
affine Lie algebras'' : we look at a family of
Lie algebras which have a lot of
similarities with affine and double affine Lie algebras, but whose structure
depends on a finite cyclic group $\Gamma$. 

More precisely, we will be interested
in the Lie algebras $\mfsl_n(\C[u]\rtimes\Ga)$, $\mfsl_n(\C[u^{\pm
1}]\rtimes\Ga)$,  $\mfsl_n(\C[u,v]\rtimes\Ga)$, $\mfsl_n(\C[u,v]^{\Ga})$,  $\mfsl_n(\C[u^{\pm
1},v]\rtimes\Ga)$, $\mfsl_n(\C[u^{\pm 1},v^{\pm 1}]\rtimes\Ga)$,
$\mfsl_n(A_1\rtimes\Ga)$ (where $A_1$ is the first Weyl algebra), $\mfsl_n(\msH_{t,\mbc}(\Ga))$ (where $\msH_{t,\mbc}(\Ga)$ is a rank one rational Cherednik algebra) and their
universal central extensions. This is motivated by the recent work
\cite{Gu3} in which deformations of the enveloping algebras of some Lie algebras closely related to these were constructed and connected to symplectic
reflection algebras
for wreath products via a functor of Schur-Weyl type. When $\Ga$ is
trivial, such deformations in the case of $\C[u^{\pm 1}]$ are the affine 
quantum groups, whereas the case $\C[u]$ corresponds to
Yangians. In the double affine setup, the quantum algebras attached to
$\C[u^{\pm 1},v^{\pm 1}],\C[u^{\pm 1},v]$ and $\C[u,v]$ for
$\mfsl_n$ are the quantum toroidal algebras, the affine Yangians and the deformed double
current algebras \cite{Gu1,Gu2}. 

In
this article, we want to study more the structure and representation
theory for the Lie algebras above, hoping that, in a future work, we
will be able to extend some of our results to the deformed setup. We
consider the central extensions for a number of reasons: in the affine
case, the full extent of the representation theory comes into life
when the centre acts not necessarily trivially; certain
presentations of those Lie algebras are actually simpler to state for
central extensions since they involve fewer relations; some of the
results can be extended without much difficulty to those central
extensions, etc.  As vector spaces, the centers of the universal
extensions are given by certain first cyclic homology groups.

At first sight, one may be tempted to think that introducing the group $\Ga$ leads to Lie algebras which are different
from those that have interested Lie theorists since the advent of Kac-Moody Lie algebras (it was our first motivation),
but this is not entirely the case. Indeed, in the one variable case, when we consider not only ordinary polynomials
but Laurent polynomials, we prove that we get back affine Lie algebras (proposition \ref{isoloop}); this is in accordance with conjectures of V.
Kac \cite{K} and the classification obtained by V. Kac and O. Mathieu \cite{K,kac,Ma1,Ma2}. In the case of Laurent polynomials in two variables, we recover toroidal Lie algebras 
(proposition \ref{isotor}). (The mixed case $\C[u^{\pm 1},v]$ also does not yield new Lie algebras.)  
However, when we consider only polynomials in non-negative powers of the variables, we
obtain distinctly new Lie algebras (see proposition \ref{presC}). 

Another motivation comes from geometry. The loop algebra
$\mfsl_n(\C[u^{\pm 1}])$ can be viewed as the space of polynomial maps
$\C^{\times} \lra \mfsl_n$. One can also consider the affine line
instead of the torus
$\C^{\times}$. More generally, one can consider the space of regular maps
$X\lra\mfsl_n$ where $X$ is an arbitrary affine algebraic variety
\cite{FeLo}. When $X$ is two-dimensional, the most natural candidate is
the torus $\C^{\times} \times \C^{\times}$, although a simpler case is the
plane $\C^2$. The variety $X$ does not necessarily have to be smooth and
one interesting singular two-dimensional case is provided by the Kleinian
singularities $\C^2 / G$ where $G$ is a finite subgroup of $SL_2(\C)$.
We are thus led to the problem of studying the Lie algebras
$\mfsl_n(\C[u,v]^{G})$ where $\C[u,v]^{G}$ is the ring of
invariant elements
for the action of $G$. However, following one of the main ideas
explained in the introduction of
\cite{EtGi}, it may be interesting to replace $\C[u,v]^{G}$
by the smash product $\C[u,v]\rtimes G$. Moreover, we can expect the full
representation theory to come to life when we consider the universal
central extensions of $\mfsl_n(\C[u,v]^{G})$ and of 
$\mfsl_n(\C[u,v]\rtimes G )$. In \cite{FeLo}, the authors
showed that the dimension of the local Weyl modules at a point $p$ in the
case of a smooth affine variety $X$ does not depend on $p$. One goal 
is to understand Weyl modules supported at a Kleinian singularity. 

This paper is organized as follows. We will denote by $\Ga$ the group $\Z/d\Z$, whereas $G$ will be a more general finite group. After general reminders on matrix Lie algebra over rings (in particular with the example of $\mfsl_n(\C[G])$ in Section \ref{matrings}, we start 
with the affine case in Section \ref{cycALA}. We study
the structure of $\mfsl_n(\C[u]\rtimes\Ga)$ and $\mfsl_n(\C[u^{\pm
1}]\rtimes\Ga)$, obtain different types of decomposition and give presentations in
terms of generators and relations.  We prove that $\mfsl_n(\C[u^{\pm
1}]\rtimes\Ga)$ is simply the usual loop algebra $\mfsl_{nd}(\C[t^{\pm
1}])$, but with a non-standard grading. Guided by the affine setup, we prove analogous results for the double affine cases in Section \ref{cycDALA}, the representations being studied in Section \ref{repcdala}.  We consider certain highest weight modules for $ \mfsl_n( \C[u,v] \rtimes \Gamma), \mfsl_n( \C[u^{\pm 1},v] \rtimes \Gamma) , \mfsl_n( \C[u^{\pm 1},v^{\pm 1}] \rtimes \Gamma)  $ and state a criterion for the integrability of their unique irreducible quotients. We also study some Weyl modules for $\mfsl_n( \C[u,v] \rtimes G)$ and $\mfsl_n( \C[u,v]^{\Gamma})$. In the first case, we show that, contrary to what might be expected at first sight, Weyl modules are rather trivial; in the second case, we can apply results of Feigin and Loktev to derive formulas for the dimension some of the local Weyl modules and we establish a lower bound on their dimensions.  In Section \ref{matdiff}, assuming usually that $t\neq 0$, we study parabolic subalgebras of $\mfgl_n( \msH_{t,\mbc}(\Ga) )$ and construct an embedding of this Lie algebra into a Lie algebra of infinite matrices. This is useful in Section \ref{hwrep} to construct quasi-finite highest weight modules. The main result of this section is a criterion for the quasi-finiteness of the irreducible quotients of Verma modules. Further possible directions of research are discussed in Section \ref{further}.
\\

{\bf Acknowledgements :} The first author gratefully acknowledges the hospitality of the Laboratoire de Math\'ematiques de l'Universit\'e de Versailles-St-Quentin-en-Yvelines where this project was started while he was a postdoctoral researcher supported by the Minist\`ere fran\c{c}ais de la Recherche. He is also grateful for the support received from the University of Edinburgh, the University of Alberta and an NSERC Discovery Grant. S.L. was partially supported by RF President Grant N.Sh-3035.2008.2, grants RFBR-08-02-00287, RFBR-CNRS-07-01-92214, RFBR-IND-0801-91300 and the P. Deligne 2004 Balzan prize in mathematics. We thank A. Pianzola for pointing out the reference \cite{BGT}. We are grateful to I. Gordon for his comments, for his suggestions regarding corollary \ref{lbWmod2} and for pointing out the counterexample from the work of M. Haiman which we present at the end of section \ref{Weylmodinv}.

\section{Matrix Lie algebras over rings}\label{matrings}

\subsection{General results}

In this section, we present general definitions and results which will be
useful later. All algebras and tensor products are over $\C$, unless specified otherwise.

\begin{definition}
Let $A$ be an arbitrary associative algebra. The Lie algebra $\mfsl_n(A)$
is defined as the derived Lie algebra $[\mfgl_n(A),\mfgl_n(A)]$ where
$\mfgl_n(A) = \mfgl_n\otimes A$.
\end{definition}

In other words, the Lie subalgebra $\mfsl_n(A)\subset \mfgl_n(A)$ is
the sum of $\mfsl_n(\C) \ot A$ and of
the space of all scalar matrices with entries in $[A,A]$.  Thus the
cyclic homology group $HC_0(A) = \frac{A}{[A,A]}$ accounts for the
discrepancy between  $\mfsl_n(A)$ and  $\mfgl_n(A)$.

Since $\mfsl_n(A)$ is a perfect Lie algebra (that is, $[\mfsl_n(A), \mfsl_n(A)]
= \mfsl_n(A)$),
it possesses a universal
central extension $\wh{\mfsl}_n(A)$ unique up to isomorphism. The following
theorem gives a simple
presentation of $\wh{\mfsl}_n(A)$ in terms of generators and
relations.

\begin{theorem}\cite{KL}
Assume that $n\ge 3$. $\wh{\mfsl}_n(A)$ is
isomorphic to the Lie algebra generated by elements $F_{ij}(a)$
($1\leq i,j\leq n$, $a\in A$) which
satisfy the following relations :
\begin{equation*} [F_{ij}(a_1), F_{jk}(a_2)] = F_{ik}(a_1a_2), \;\;
\text{ for }i\neq
j\neq k\neq i,   \end{equation*}
\begin{equation*}  [F_{ij}(a_1),
F_{kl}(a_2)] = 0, \;\; \text{ for }i\neq j\neq k\neq l\neq i .
\end{equation*}
\end{theorem}
Here $i\neq j\neq k\neq i$ means ($i\neq j$ and $j\neq k$ and $k\neq
i$). We will use this convention in this paper.

When $n=2$, one has to add generators $H_{12}(a_1,a_2)$ given bt $H_{12}(a_1,a_2) = [F_{12}(a_1),F_{21}(a_2)]$ for $a_1,a_2 \in A$, and the relations \[
[H_{12}(a_1,a_2), F_{12}(a_3)] = F_{12}(a_1a_2a_3 + a_3a_2a_1), \;\; [H_{12}(a_1,a_2), F_{21}(a_3)] = -
F_{21}(a_3 a_1a_2 + a_2a_1a_3). \] 

It is also proved in \cite{KL} that the center of $\wh{\mfsl}_n(A)$
is isomorphic, as a vector space, to the first cyclic homology group
$HC_1(A)$. For $G$ a finite group and $A=\C[G]$, $HC_1(A)=0$, but in
the double affine case below, the center will be infinite dimensional.

The following formulas taken from \cite{VaVa2} can help understand better the bracket on
$\wh{\mfsl}_n(A)$. The problem of computing explicitly the bracket with respect to the decomposition $\wh{\mfsl}_n(A) \cong \mfsl_n(A) \oplus HC_1(A)$ is, in general, a
difficult one, but it is possible to obtain some nice formulas by using a different
splitting of $\wh{\mfsl}_n(A)$. Let $\lan A,A\ran$ be the quotient
of $A\ot A$ by the two-sided ideal generated by $a_1\ot a_2 -
a_2\ot a_1$ and $a_1a_2\ot a_3 - a_1\ot a_2 a_3 - a_2\ot a_3 a_1$. The first cyclic
homology group $HC_1(A)$ is, by definition, the kernel of
the map $\lan A,A\ran \onto [A,A], \, a_1 \ot a_2 \mapsto [a_1,a_2]$.

For $m_1,m_2\in\mfsl_n$, $a_1,a_2\in A$, and $(\cdot,\cdot)$ the Killing form on
$\mfsl_n$, set :
$$[m_1,m_2]_+ = m_1m_2 + m_2m_1 - \frac{2}{n}(m_1,m_2)I\text{ ,
}[a_1,a_2]_+ = a_1a_2 + a_2a_1.$$

\begin{proposition}\cite{VaVa2}\label{UCE} The Lie algebra $\wh{\mfsl}_n(A)$ is isomorphic to the
vector space
$\mfsl_n \ot A \oplus \lan A,A\ran $ endowed with the bracket:
\begin{equation*}  [ m_1\ot a_1, m_2\ot a_2 ] =
\frac{1}{n}(m_1,m_2)\lan a_1,a_2\ran + \frac{1}{2}[m_1,m_2]\ot
[a_1,a_2]_+ + \frac{1}{2}[m_1,m_2]_+ \ot [a_1,a_2], \label{UCE1}
\end{equation*}
\begin{equation*}  [\lan a_1,a_2\ran, \lan b_1,b_2 \ran ] = \lan
[a_1,a_2], [b_1,b_2] \ran,  \label{UCE2} \end{equation*}
\begin{equation*} [ \lan a_1,a_2\ran , m_1\ot a_3 ] = m_1 \ot
[[a_1,a_2],a_3].  \label{UCE3}  \end{equation*}
\end{proposition}

\subsection{Example : special linear Lie algebras over group rings}

Let $G$ be a finite group. One interesting case for us is the group
algebra $A=\C[G]$, in which case $HC_0(A) \cong
\C^{\oplus cl(G)}$ where $cl(G)$ is the number of conjugacy classes of
$G$.

\begin{lemma} The Lie algebra $\mfsl_n(\C[G])$ is semi-simple of Dynkin type $A$.
\end{lemma}

\demo Recall that $cl(G)$ is also the number of irreducible representations of $G$.
Enumerate the irreducible representations of $G$ by $\rho_1, \ldots, \rho_{cl(G)}$
and let $d(j)$ be the dimension of $\rho_j$. Wedderburn's theorem states that the group algebra $\C[G]$ is isomorphic to
$\bigoplus_{j=1}^{cl(G)} M_{d(j)}$ as algebras where $M_{d(j)}$ is the associative algebra of $d(j)
\times d(j)$-matrices. Therefore, $\mfsl_n(\C[G]) \cong \bigoplus_{j=1}^{cl(G)}
\mfsl_{nd(j)}$.
\qed

\begin{remark} The direct sum above is a direct sum of Lie algebras, that is, two
different copies of $\mfsl_n$ commute. A non-degenerate symmetric invariant bilinear form $\ka$ on the semisimple Lie algebra $\mfsl_n(\C[G])$ is given
by the formula ($m_1,m_2\in \mfsl_n$, $\ga_1,\ga_2\in G$) :
$$\ka(m_1\ga_1,m_2\ga_2) = Tr(m_1\cdot
m_2)\delta_{\ga_1=\ga_2^{-1}}.$$ \end{remark}

\section{Cyclic affine Lie algebras}\label{cycALA}

For an arbitrary ring $R$ with action
of a finite group $G$, $R\rtimes G$ is the ring spanned by the
elements $ag, a\in R, g\in G$ with the relations $(a_1g_1)
\cdot (a_2g_2) = a_1g_1(a_2)g_1g_2$. In the previous section,
just by considering group rings (over $\C$), we ended up with
semi-simple Lie algebras. Here, when $R$ is a Laurent polynomial ring,
one can expect to obtain affine Kac-Moody algebras, which is indeed
what happens.

\subsection{Definition and decomposition}

Let $\xi$ be a generator of $\Ga$ and $\zeta$ a
primitive $d^{th}$-root of unity. Let $A = \C[u^{\pm
1}]\rtimes\Gamma$
and $B = \C[u]\rtimes\Gamma$. The action of $\Gamma$ is defined by
$\xi(u) = \zeta u$. We will be interested in
the structure of the Lie algebra $\mfsl_n(A)$
and of its universal central extension $\wh{\mfsl}_n(A)$. We will also
say a few words about
$\mfsl_n(B)$.  We will show that
$\mfsl_n(A)$ is a graded simple Lie algebra
and explain how it is related to the classification of such Lie
algebras obtained by V. Kac and O. Mathieu \cite{K,kac,Ma1,Ma2}.

In the following, $\equiv$ is the equivalence mod $d$.

\begin{lemma} We have :
$$[A,A] = \bigoplus_{i=1}^{d-1} \C[u^{\pm 1}] \xi^i
\oplus \bigoplus_{j \in \Z, j\not\equiv 0}
\C\cdot u^j\text{ and }[B,B] = \bigoplus_{i=1}^{d-1} u\C[u] \xi^i
\oplus \bigoplus_{j\Z_{\ge 1} , j\not\equiv 0} \C\cdot u^j .$$
\end{lemma}
\begin{proof}  If $1\le i\le d-1$, then $u^j \xi^i =
\frac{1}{1-\zeta^i}[u,u^{j-1}\xi^i]$. If $j\in \Z$ and $j \not\equiv
0 $, then $u^j =
\frac{1}{\zeta^j-1}[\xi,u^j\xi^{-1}]$. This proves $\supseteq$. Consider $[u^{k}\xi^a, u^m\xi^b] = (\zeta^{am} - \zeta^{bk}) u^{k+m}\xi^{a+b}$ and suppose that the right-hand side is in $\C[u^{\pm 1}]^{\Ga}$. Then $k+l\equiv 0$ and $a+b\equiv 0 \, \mathrm{mod}\,(d)$, so $\zeta^{am} - \zeta^{bk}=0$. This proves $\subseteq$. \end{proof}

\begin{corollary} $HC_0(B)  \cong \C[u]^{\Ga} \oplus \C[\Ga]$.  \end{corollary}

The Lie algebra $\mfsl_n(A)$ admits different
vector space decompositions, similar to the two standard triangular
decompositions of affine Lie algebras. Let $\mfn^+$ (resp. $\mfn^-$)
be the Lie algebra of strictly upper (resp. lower) triangular matrices
in $\mfsl_n$ (over $\C$) and let $\mfh$ be the usual Cartan subalgebra of
$\mfsl_n$. The elementary matrices in $\mfgl_n$ will be denoted
$E_{ij}$ and $I$ will stand for the identity matrix. In the following,
by abuse of notation, an element $g\otimes a$ will be denoted $ga$ for
$g\in \mfgl_n$, $a\in A$. We have the following
vector space isomorphisms (triangular decompositions) :
$$\mfsl_n(A) \cong ( \mfn^- A ) \oplus (
\mfh A \oplus I [A,A] )  \oplus ( \mfn^+ A  ), \label{1td}
$$
and $\mfsl_n(A)$ is also isomorphic to the sum
\begin{equation*}
\begin{split}
&\left(\mfsl_n\ u^{-1}\C[u^{-1}]\rtimes\Ga \oplus \left(
\bigoplus_{\stackrel{j\le -1,j\not\equiv 0}{0\leq i\le d-1}}
\C I u^j \xi^i   \oplus
\bigoplus_{1\leq i\leq d-1} I u^{-d}\C[u^{-d}]\xi^i \right) \oplus
\mfn^- \C[\Ga]
\right) \\
\oplus & \left(\mfh \C[\Ga]\oplus \bigoplus_{1\leq i\leq d-1}
\C I \xi^i \right)
\\\oplus &\left( \mfsl_n u\C[u]\rtimes\Ga \oplus \left(
\bigoplus_{\stackrel{j\ge 1, j\not\equiv 0}{0\le i\le d-1}}
\C I u^j \xi^i  \oplus \bigoplus_{1\le i \le d-1} I u^d\C[u^d]\xi^i
\right)  \oplus  \mfn^+  \C[\Ga]  \right).  \label{2td}
\end{split}
\end{equation*} These lead to similar decompositions for  $\mfsl_n(B)$. These triangular decompositions are similar to those considered, for instance, in \cite{Kh}.

The first triangular decomposition is analogous to the loop triangular
decomposition of affine Lie algebras, but in our situation the middle
Lie algebra is not commutative. The second triangular decomposition is
similar to the decomposition of affine Lie algebras adapted to
Chevalley-Kac generators, and it is of particular importance as the
middle term $\mfH$ in commutative. So the role of the Cartan
subalgebra will be played by the commutative Lie algebra $\mfH$ and
our immediate aim is to obtain a corresponding appropriate root space
decomposition of $\mfsl_n(A)$.

It will be
convenient to work with the primitive idempotents of $\Ga$, so let us
set $\mbe_j = \frac{1}{d} \sum_{i=0}^{d-1}\zeta^{-ij}\xi^i$. A vector
space basis of $\mfH$ is given by the following elements:
$$H_{i,j} = 
\begin{cases}
(E_{i,i}-E_{i+1,i+1})\mbe_j&\text{ for $1\le i\le n-1, 0\le j\le
d-1$,}
\\E_{n,n} \mbe_{j} - E_{1,1} \mbe_{j+1}&\text{ for $i = 0$, $0\le j\le d-2$.}
\end{cases}
$$
Remark : we could define $H_{0,d-1}$ in the same way, but then we  would get
$\sum_{i=0}^{n-1}\sum_{j=0}^{d-1}H_{i,j} = 0$ in $\mfsl_n(A)$ (but lifts to a non-zero
central element in $\wh{\mfsl}_n(A)$).

\begin{lemma} A basis of the eigenspaces for non-zero eigenvalues for the adjoint action of
$\mfH$ on
$\mfsl_n(A)$ (except for $\mfH$ itself) is given by the following vectors:
\[ E_{ij} u^k \mbe_l \;\; 1\le
i\neq j\le n, k\in\Z, 0\le l\le d-1,\]
\[E_{ii} u^k\mbe_l \;\;
1\le i\le n, k\not\equiv 0, 0\le l\le d-1.\]
\end{lemma}

\begin{proof} This is a consequence of the following simple computations:
For $1\le i\neq j\le n, 0\le a\le n-1, k\in\Z, 0\le l\le d-1, 0\le b\le d-1$,
$$[H_{a,b}, E_{ij} u^k
\mbe_l ]  =  
\begin{cases} \big( \delta_{b \equiv l}(\delta_{a+1\equiv j} - \delta_{a\equiv j})
+ \delta_{b-k \equiv l}(\delta_{a\equiv i} - \delta_{a+1\equiv i}) \big)
E_{ij} u^k\mbe_l&\text{ if $a\neq 0$,}
\\(
\delta_{n\equiv i}\delta_{b-k \equiv l} - \delta_{n\equiv j}\delta_{b
\equiv l} - \delta_{1\equiv i}\delta_{b + 1 -k \equiv l} - \delta_{j\equiv
1}\delta_{b + 1 \equiv l} )   E_{ij} u^k \mbe_l&\text{ if $a = 0$.}
\end{cases}
$$
For $1\le i\le n, k \not\equiv 0, 0\le l\le d-1$,
$$ [H_{a,b}, E_{ii}  u^k\mbe_l ]   =  
\begin{cases}
\big( \delta_{a\equiv i} -
\delta_{a+1\equiv i}\big) \big( \delta_{b-k\equiv l}-\delta_{b\equiv
l}\big) E_{ii}  u^k \mbe_l &\text{ if $a\neq 0$,}
\\ \big( \delta_{n\equiv i} \big(\delta_{b-k \equiv l} - \delta_{b \equiv l} \big)  -
\delta_{i \equiv 1} ( \delta_{b + 1 -k \equiv l} -
\delta_{b +1 \equiv l} ) \big) E_{ii}  u^k \mbe_l &\text{ if $a = 0$.}
\end{cases}
$$

\end{proof}

\subsection{Derivation element and imaginary roots}

We will introduce the real, imaginary roots and root spaces later
after adding a derivation $\msd$ to $\wh{\mfsl}_n(A)$.

The center of the universal central extension $\wh{\mfsl}_n(A)$ of
$\mfsl_n(A)$ is
isomorphic to $HC_1(A)$ \cite{KL}. It is known
that \[  HC_1(A) \cong \frac{(\C[u^{\pm
1}]du)^{\Ga}}{d(\C[u^{\pm 1}]^{\Ga})},  \] the quotient of the space
of $\Ga$-invariant $1$-forms on $\C^{\times}$ by the space of exact
$1$-forms coming from $\Ga$-invariant Laurent polynomials. This can be
deduced from the isomorphism $A\cong M_d(\C[t^{\pm 1}])$ - see proposition \ref{isoloop}. Thus this cyclic homology group is
one dimensional, with basis given by $u^{-1}du$, which we denote as
usual by $\msc$.

\begin{definition} The cyclic affine Lie algebra $\ol{\mfsl}_n(A)$ is
obtained from $\wh{\mfsl}_n(A)$ by adding a
derivation $\msd$ that satisfies the relations $ [\msd,
E_{ij} u^{k} \mbe_l] = k E_{ij} u^{k} \mbe_l$.
\end{definition}

Set $\ol{\mfH} =
\mfH \oplus \C\cdot \msc \oplus \C\cdot \msd$. We can now introduce
the roots as appropriate elements of
$$\ol{\mfH}^*_0 = \{\lambda\in \ol{\mfH}^*|\lambda(\msc) = 0\}.$$
The real root spaces
are spanned by the root vectors $E_{ij} u^k \mbe_l$ where $0\le l\le d-1$, 
$1\le i, j\le n$, $k\in\Z$, with the condition $k\not\equiv 0$ if $i = j$.

\noindent The imaginary root spaces
are spanned by the following root vectors:
$$ H_{i}  u^{kd} \mbe_l\text{ for } 1\le i\le n-1, k \neq 0,
0\le l\le d-1,$$
$$ E_{nn}  u^{kd}\mbe_l - E_{11} u^{kd}\mbe_{l+1}\text{ for } k\neq 0, 0\le l\le d-2.$$

     We want to identify the root lattice as a lattice in
$\ol{\mfH}^*_0$. Let us introduce the elements $\eps_{i,l}\in
\ol{\mfH}^*_0, 1\le i\le n, 0\le l\le d-1$ by setting
$$\eps_{i,l}(H_{a,b}) = (\delta_{a= i} - \delta_{a+1= i}
)\delta_{b \equiv l},\;  \eps_{i,l}(H_{0,b}) = \delta_{n = i} \delta_{b \equiv l} -
\delta_{i = 1}\delta_{b +1 \equiv l},$$
and $\eps_{i,l}(\msd)=0$.

\begin{definition}The
real roots are $\eps_{i,k+l} - \eps_{j,l} + k\delta$ for $1\le i,
j\le n,  0 \le l \le d-1, k\in\Z$ with $i\neq j$ or, if $i=j$, then $k\equiv 0$; the imaginary ones are $kd\delta$
where $\delta(\msd)=1, \delta(H_{a,b})=0$ and $k\neq 0$.
\end{definition}
They generate a
lattice - the root lattice - in $\ol{\mfH}^*_0$.  As one can verify,
the real root spaces all have dimension one.

\begin{lemma} The root lattice is freely generated by the following
roots, which we
will call the positive simple roots:  $ \eps_{i,l} - \eps_{i+1,l} $
for $1\le i\le n-1,  0\le l \le d-1$, $ \eps_{n,l} - \eps_{1,l+1} $ for $0\le l\le d-2$,
and $(\eps_{n,d-1} -
\eps_{1,0}) + \delta$. \end{lemma}

\begin{proof} Indeed, note that  \[ \delta = \big( (\eps_{n,d-1} -
\eps_{1,0}) + \delta \big) +  \sum_{l=0}^{d-1}   \sum_{i=1}^{n-1} (\eps_{i,l} - \eps_{i+1,l} )
+ \sum_{l=0}^{d-2} (\eps_{n,l} - \eps_{1,l+1} )  \]
The set of simple roots contains $nd$ elements, which is also the dimension of $\ol{\mfH}^*_0$.
\end{proof}

\subsection{Cyclic affine Lie algebras and affine Lie algebras}

\begin{proposition}\label{grsimple}
The Lie algebra $\mfsl_n(A)$ is graded
simple (i.e. it contains no non trivial graded ideal).
\end{proposition}

\begin{proof}
Suppose that $\wt{I} = \sum_{m\in\Z} \wt{I}_m$ is a non-zero graded ideal of
$\mfsl_n(\C[u^{\pm
1}]\rtimes\Gamma)$. Then $\wt{I}$ is stable under the adjoint action of $\mfH$,
hence each graded piece $\wt{I}_m$ must decompose into the direct sum of
all the root spaces contained
in $\wt{I}_m$. It can be checked that the ideal generated by any real root vector
is the whole $\mfsl_n(A)$, so if $\wt{I}$ contains
a real root vector, then $\wt{I}$ is the whole Lie algebra. Moreover, if $\wt{I}$ contains
an imaginary root vector, then it contains also a real one.
\end{proof}

However, the Lie algebra $\mfsl_n(A)$
is not simple if we do not take the grading into account, as can be
seen from proposition
\ref{isoloop} below. A conjecture of V. Kac \cite{K}, proved in general by
O. Mathieu \cite{Ma1,Ma2}, gives a classification of graded simple
Lie algebras of polynomial growth according to which, following
proposition \ref{grsimple}, $\mfsl_n(A)$ must
be isomorphic to a (perhaps twisted) loop algebra. This is indeed the case,
although the isomorphism does not respect the natural grading on the
loop algebra $\mfsl_{nd}\ot \C[t^{\pm 1}]$.

\begin{proposition}\label{isoloop}
The Lie algebras $\mfsl_n(A)$ and
$\mfsl_{nd}(\C[t^{\pm 1}])$ are isomorphic.
\end{proposition}

\begin{proof}
An isomorphism is given explicitly by the following formulas: If $0\leq l\leq d-1$,
$k\in\Z$,
$-l\le r\le d-l-1$, $1\le i\neq j\le n$,
\[
\begin{split}E_{n(l+r)+i, nl+j} t^k
  & \leftrightarrow   E_{ij} u^{kd+r}\mbe_l,
\\ E_{n(l+r)+i,nl+i}
t^k  & \leftrightarrow  E_{ii}  u^{kd+r} \mbe_l \mbox{ for } r\neq 0,
\\(E_{nl+i,nl+i}-E_{nl+i+1,nl+i+1})  t^k & \leftrightarrow H_i  u^{kd}
\mbe_l \mbox{ for } i\neq n,
\\(E_{nl,nl} -
E_{nl+1,nl+1})  t^k & \leftrightarrow  E_{nn} u^{kd}\mbe_l - E_{11}
u^{kd}\mbe_{l+1} \mbox{ for } l\neq 0.
\end{split} \]
These formulas can be obtained via the algebra isomorphism $A \cong M_d(\C[t^{\pm 1}])$ which is given by $E_{l+r,l} t^k  \leftrightarrow u^{kd+r} \mbe_l$. \end{proof}

Let $\phi:\mfsl_{nd}(\C[t^{\pm 1}]) \iso \mfsl_n(\C[u^{\pm 1}]\rtimes\Ga)$ be the isomorphism
in the proof of proposition \ref{isoloop}. We can put a grading on $\mfsl_{nd}$ by giving
$E_{ij}$ degree $j-i$. This induces another grading on $\mfsl_{nd}(\C[t^{\pm 1}])$,
besides the one coming from the powers of $t$: the total of these two is the grading given by $deg(E_{ij}(t^r))=j-i+r$. 

\begin{proposition} $\phi$ is an isomorphism of graded Lie algebras when $\mfsl_{nd}(\C[t^{\pm 1}])$ is endowed with the total grading and $\mfsl_n(\C[u^{\pm 1}]\rtimes\Ga)$ is graded in terms of powers of $u$.
\end{proposition}

Let $\mfp$ be the parabolic subalgebra of $\mfsl_{nd}$ consisting of all
the lower triangular matrices and the matrices with $d$ blocks of $n\times n$
matrices along the diagonal. Let $\mfn$ be the nilpotent subalgebra
which consists of all the upper triangular matrices except those in
these blocks along the diagonal, so that $\mfsl_{nd} \cong \mfp \oplus
\mfn$.  The isomorphism given in proposition \ref{isoloop} shows that
$\mfsl_n(B)$ is isomorphic to $\mfp \ot \C[t]
\oplus \mfn \ot t\C[t]$.

We gave in the previous subsection two triangular decompositions of $\mfsl_n(A)$.
The Lie algebra $\mfsl_{nd}(\C[t^{\pm 1}])$ admits
similar decompositions, namely:
\[ \mfsl_{nd}(\C[t^{\pm 1}]) \cong \mfn_{nd}^- \ot \C[t^{\pm 1}]
\bigoplus \mfh_{nd} \ot \C[t^{\pm 1}]
\bigoplus \mfn_{nd}^+ \ot \C[t^{\pm 1}],\]
\[ \mfsl_{nd}(\C[t^{\pm 1}]) \cong \big( \mfsl_{nd}\ot t^{-1}\C[t^{-1}] \oplus
\mfn_{nd}^- \big)
\bigoplus \mfh_{nd}  \bigoplus   \big( \mfn_{nd}^+ \oplus \mfsl_{nd}\ot t\C[t] \big)  .\]
The isomorphism given in proposition \ref{isoloop} preserves the
second decomposition, but not the first one.

We conclude that a lot is already known about the representation theory of cyclic
affine algebras and simple finite dimensional representations are classified by
tuples of $nd-1$ polynomials (see \cite{ca}). We state this more explicitly for toroidal
Lie algebras in section \ref{inthwmod} below.

\section{Cyclic double affine Lie algebras}\label{cycDALA}

In this section, we set $A=\C[u^{\pm 1},v^{\pm 1}]\rtimes\Gamma, B=\C[u^{\pm
1},v]\rtimes\Gamma, C=\C[u,v]\rtimes\Gamma$. Here, $\xi$ acts on $u,v$
by $\xi(u) = \zeta u, \xi(v) =
\zeta^{-1}v$. Note that, setting $w=uv$, we deduce that $A$ is isomorphic
to $\C[u^{\pm 1},w^{\pm 1}]\rtimes\Ga$ where $\Ga$ acts trivially on $w$.
The same remark applies to $B$.

We  will be interested in the structure of the Lie
algebras $\mfsl_n(A)$, $\mfsl_n(B)$, $\mfsl_n(C)$ and of their
universal central extensions.

\subsection{Structure}

    As in the previous section, we need to know certain cyclic homology
groups.
\[ HC_0(A) = A^{\Ga}  \text{ , }  HC_0(B) = B^{\Ga} \text{ , }HC_0(C)
= C^{\Ga}\oplus \C^{\oplus(|\Ga|-1)},\]
$HC_1(A) \cong
\frac{\Omega^1(A)^{\Ga}}{d(A^{\Ga})}$ and similarly
for $B,C$ (see corollary \ref{HC} below.). As vector spaces, when $\Ga \neq \{ \mathrm{id}\}$, we
have
$$
HC_1(A) = v\C[u^{\pm 1},v^{\pm 1}]^{\Ga}du  \oplus  \C[u^{\pm 1}]^{\Ga}
v^{-1}dv 
   =  (\C[u^{\pm 1}]^{\Ga} \ot_{\C}\C[w^{\pm 1}]) u^{-1}du \oplus
\C[u^{\pm 1}]^{\Ga}w^{-1}dw, $$
\[ HC_1(B) = v\C[u^{\pm 1},v]^{\Ga}du
\oplus \C u^{-1}du \cong (w\C[w] \ot_{\C} \C[u^{\pm 1}]^{\Ga}) u^{-1} du \oplus \C u^{-1}
du,  \]
\[HC_1(C) = v\C[u,v]^{\Ga}du.\] These can be
obtained from the computations of Hochschild homology in \cite{Fa}, which are
valid for more general groups than $\Ga$. For $A$ and $B$, they can also
be deduced from the proof of the following extension of proposition \ref{isoloop}.

\begin{proposition}\label{isotor}
The Lie algebra $\mfsl_n(A)$ (resp. $\mfsl_n(B)$) is isomorphic to the toroidal Lie
algebra $\mfsl_{nd}(\C[s^{\pm 1},t^{\pm 1}])$ (resp.
$\mfsl_{nd}(\C[s^{\pm 1},t])$).
\end{proposition}

\begin{proof}
We write $A$ as $A=\C[u^{\pm 1}, w^{\pm 1}]$. Since $\C[u^{\pm 1}]\rtimes
\Ga \cong M_d(\C[s^{\pm 1}])$ (with $s=u^d$, see the proof of
proposition \ref{isoloop}) and $A\cong (\C[u^{\pm
}]\rtimes\Ga) \ot_{\C} \C[w^{\pm 1}]$, we immediately deduce this
proposition (setting $t=w$). The same argument applies to $B$.
\end{proof}

\begin{remark}
The isomorphism in this proposition is reminiscent of lemma 4.1 in \cite{BGT}. Furthermore, in that article, using vertex operator techniques, the authors constructed representations of a certain affinization of $\mfgl_N(\C[G][t,t^{-1}])$, where $G$ is an admissible subgroup of $\C^{\times}$ (in the sense defined in that article).
\end{remark}

Explicitly, the isomorphism in proposition \ref{isotor} sends $E_{ab}u^i w^j \mbe_k$ with $i=m+dl,
-k\le m \le d-1-k, 0\le k\le d-1$ to $E_{a+(m+k)n,b+kn} s^l t^j$; in
terms of $u,v$ instead of $u,w$, it maps $E_{ab}u^i v^j \mbe_k$ with
$i-j=m+dl, -k\le
m\le d-1-k$ to $E_{a+(m+k)n,b+kn} s^l t^j$. In particular, if $i=0$, then
this map restrict to the same isomorphism as in proposition \ref{isoloop}
for $\C[v^{\pm 1}]$ with $v$ playing the role of $u^{-1}$. The Lie
subalgebra $\mfsl_n(\C[w^{\pm 1}]\rtimes\Ga)$ gets identified with the
direct sum of $d$ copies of $\mfsl_n(\C[t^{\pm 1}])$, which agrees with
$\C[w^{\pm 1}]\rtimes \Ga \cong \C[w^{\pm 1}]\ot_{\C}\C[\Ga] \cong
\C[w^{\pm 1}]^{\oplus d}$.

\begin{corollary}\label{HC}
The cyclic homology groups $HC_1(A)$ and $HC_1(B)$ are given by \[ HC_1(A)
\cong \frac{\Omega^1( \C[s^{\pm 1},t^{\pm 1}]) }{d(\C[s^{\pm 1}, t^{\pm 1}])}
\text{ , } HC_1(B) \cong \frac{\Omega^1(\C[s^{\pm 1},t])}{d(\C[s^{\pm 1},
t])}.\]
\end{corollary}

\begin{proof}
This is a corollary of the algebra isomorphism $A \cong M_d(\C[s^{\pm 1},t^{\pm 1}]) )$ in the proof of proposition \ref{isotor}. (A similar isomorphism holds for $B$.)
\end{proof}

When we restrict the isomorphism of proposition \ref{isotor} to
$\mfsl_n(C)$, we obtain an injective map from $\mfsl_n(C)$ into
$\mfsl_{nd}(\C[s,t])$. It comes from an injection of $C$ into
$M_d(\C[s,t])$. This latter map is a special case ($n=1,x,y,\nu=0$) of
the homomorphism
introduced in subsection 6.1 in \cite{Go2} - see also \cite{CB}. It is
not an epimorphism from
$C$ to $M_d(\C[s,t])$ in the set theoretical sense, but it is an
epimorphism in the following categorical sense: any two ring
homomorphisms $M_d(\C[s,t]) \stackrel{\lra}{\lra} D$ whose composites
with  $C \into M_d(\C[s,t])$ agree on $C$ must be equal.

Using proposition \ref{isotor} and the explicit isomorphism given right
after the proof, we can give a formula for the bracket on
$\wh{\mfsl}_n(A)$. It is easier to do it first with $u,w$ and then
translate the result to $u,v$. Note that we identify the center of
$\wh{\mfsl}_n(A)$ with $\frac{\Omega^1(\C[s^{\pm 1},t^{\pm 1}])}{d(\C[s^{\pm
1}, t^{\pm 1}])}$, so that it makes sense to write $d(f)$ for some
$f\in\C[s^{\pm 1},t^{\pm 1}]$ (with $s=u^d,t=w$ as above). For a
commutative algebra $R$, the bracket on $\mfsl_n(R)$ is given by
$[E_{a_1b_1}r_1,E_{a_2b_2}r_2] = [E_{a_1b_1},E_{a_2b_2}]r_1r_2 +
\mathrm{Tr}(E_{a_1b_1}E_{a_2b_2}) r_1 d(r_2)$, where $\mathrm{Tr}$ is
the usual trace functional \cite{Ka} and $r_1 d(r_2) \in \frac{\Omega^1(R)}{dR}$.
Now consider 
\[ \begin{split}
&[E_{a_1+(m_1+k_1)n, b_1+k_1n} s^{l_1}t^{j_1}, E_{a_2+(m_2+k_2)n,
b_2+k_2n} s^{l_2}t^{j_2}]
\\= & \delta_{b_1+k_1n = a_2+(m_2+k_2)n} E_{a_1+(m_1+k_1)n, b_2+k_2n}
s^{l_1+l_2}t^{j_1+j_2}
\\& - \delta_{b_2+k_2n = a_1+(m_1+k_1)n}E_{a_2+(m_2+k_2)n,
b_1+k_1n}s^{l_1+l_2}t^{j_1+j_2}
\\&+
\delta_{a_1+(m_1+k_1)n=b_2+k_2n} \delta_{b_1+k_1n = a_2 + (m_2+k_2)n}
s^{l_1}t^{j_1} d(s^{l_2}t^{j_2}).
\end{split} \]

The isomorphism of proposition \ref{isotor} identifies the left-hand side
with
$$[ E_{a_1b_1}u^{i_1}w^{j_1}\mbe_{k_1},
E_{a_2b_2}u^{i_2}w^{j_2}\mbe_{k_2}], \;\; i_1= dl_1 + m_1 ,\, i_2= dl_2 + m_2 ,$$
whereas the right-hand side is
identified with  \[ \delta_{a_2=b_1}\delta_{k_1\equiv m_2+k_2}
E_{a_1,b_2}u^{i_1+i_2}w^{j_1+j_2} - \delta_{b_2=a_1}\delta_{m_1+k_1\equiv
k_2}E_{a_2b_1}u^{i_1+i_2} w^{j_1+j_2} \] \[ +
\delta_{a_1=b_2}\delta_{m_1+k_1\equiv
k_2}\delta_{b_1=a_2}\delta_{k_1\equiv
m_2+k_2}u^{i_1-m_1}w^{j_1}d(u^{i_2-m_2}w^{j_2}).\]
Setting $v=wu^{-1}$, we obtain a formula in terms of $u$ and $v$, if
we now define $m_1,m_2$ by $i_i-j_i = m_i+l_i d$:
\begin{equation*}
\begin{split}
&[ E_{a_1b_1}u^{i_1}v^{j_1}\mbe_{k_1},
E_{a_2b_2}u^{i_2}v^{j_2}\mbe_{k_2}]
\\ = & \delta_{a_2=b_1}\delta_{k_1\equiv
m_2+k_2} E_{a_1,b_2}u^{i_1+i_2}v^{j_1+j_2}
    - \delta_{b_2=a_1}\delta_{m_1+k_1\equiv k_2}E_{a_2b_1}u^{i_1+i_2}
v^{j_1+j_2} \\
    & + \delta_{a_1=b_2}\delta_{m_1+k_1\equiv
k_2}\delta_{b_1=a_2}\delta_{k_1\equiv m_2+k_2}
u^{i_1-m_1}v^{j_1}d(u^{i_2-m_2}v^{j_2}).
\end{split}
\end{equation*}

The formula for the bracket when $A$ is replaced by $B$ and $C$ is exactly
the same, simply obtained by restricting the values allowed for
$i_1,i_2,j_1,j_2$.  The description of $\wh{\mfsl}_n(A)$ given by proposition \ref{UCE} implies that the natural maps $\wh{\mfsl}_n(B),\wh{\mfsl}_n(C) \lra \wh{\mfsl}_n(A)$ are embeddings. 

The next proposition is an adaptation of proposition 3.5 in \cite{MRY} to
$\wh{\mfsl}_n(A)$, using proposition \ref{isotor}. There are a few more relations in
proposition
\ref{DALApres} below than in \textit{loc. cit.}, but the proofs are the
same.

Let $\mathbf{C}=(c_{ij})$ be the affine Cartan matrix of type $\wh{A}_{nd-1}$ with
rows and columns indexed from $0$ to $nd-1$. Let $f:[0,n-1]\times
[0,d-1]\lra [0,nd-1]$ be the function $f(i,j) = i+jn$.

\begin{proposition}\label{DALApres}
The Lie algebra $\wh{\mfsl}_n(A)$ is isomorphic to the
Lie algebra $\mft$ generated by the elements $X_{i,j,r}^{\pm},
H_{i,j,r}$ for $0\le i\le n-1, 0\le j\le d-1, r\in\Z$, and
a central element $\msc$ satisfying the following list of
relations.
\[ [H_{i_1,j_1,r_1}, H_{i_2,j_2,r_2}] = r_1 c_{f(i_1,j_1),f(i_2,j_2)}
\delta_{r_1+r_2=0} \msc, \]
\[[H_{i_1,j_1,0}, X_{i_2,j_2,r_2}^{\pm}] =
\pm c_{f(i_1,j_1),f(i_2,j_2)}  X_{i_2,j_2,r_2}^{\pm}, \]
   \[ [H_{i_1,j_1,r_1+1}, X_{i_2,j_2,r_2}^{\pm}]  =
[H_{i_1,j_1,r_1}, X_{i_2,j_2,r_2+1}^{\pm}],\]
\[[X_{i_1,j_1,r_1+1}^{\pm}, X_{i_2,j_2,r_2}^{\pm}]  =
[X_{i_1,j_1,r_1}^{\pm}, X_{i_2,j_2,r_2+1}^{\pm}], \]
\[ [X_{i_1,j_1,r_1}^+, X_{i_2,j_2,r_2}^{-}] =
\delta_{i_1=i_2}\delta_{j_1=j_2} (H_{i_1,j_1,r_1+r_2} + r_1
\delta_{r_1+r_2=0}\msc),  \]
\[
ad(X_{i_1,j_1,r_1}^{\pm})^{1-c_{f(i_1,j_1),f(i_2,j_2)}}(X_{i_2,j_2,r_2}^{\pm})
= 0  .\]

\end{proposition}

\begin{remark}
The elements with $r=0$ generate a copy of $\wh{\mfsl}_n(\C[u^{\pm
1}]\rtimes\Ga)$ and this proposition gives a set of relations describing this
algebra, which is the one in terms of Chevalley-Kac generators of
$\wh{\mfsl}_{nd}(\C[t^{\pm 1}])$ - see proposition \ref{isoloop}. The elements
with $i\neq 0$ generate a central extension of $\mfsl_n(\C[w^{\pm
1}])^{\oplus d} $ and this proposition gives a presentation of
$\mfsl_n(\C[w^{\pm 1}])$. (This is the one obtained by considering a Cartan matrix of finite type $A_{n-1}$ in proposition 3.5 in \cite{MRY}.)
\end{remark}

An isomorphism $\tau: \mft \iso \wh{\mfsl}_n(A)$ is given explicitly on the
generators by the following formulas: \\

$$(X_{i,j,r}^{+}, X_{i,j,r}^{-})
\mapsto \begin{cases}
(E_{i,i+1} w^r \mbe_j , E_{i+1,i} w^r\mbe_j )&\text{ if $i\neq 0$,}
\\(E_{n1} u^{-1} w^r \mbe_j, E_{1n} u w^r \mbe_{j-1} ) &\text{ if $i = 0,j\neq 0$,}
\\(E_{n1} u^{2d-1} w^r  \mbe_{d-1} , E_{1n} \ot u^{-(2d-1)} w^r \mbe_0)&\text{ if $i = j = 0$,}
\end{cases}$$
$$H_{i,j,r}
\mapsto \begin{cases}
(E_{ii} - E_{i+1,i+1})w^r \mbe_j&\text{ if $i\neq 0$,}
\\E_{nn} w^r \mbe_{j-1} - E_{11} w^r \mbe_j &\text{ if $i = 0,j\neq 0$,}
\\E_{nn} w^r \mbe_{0} - E_{11} w^r \mbe_{d-1} - d w^r u^{-1}du&\text{ if $i = j = 0$,}
\end{cases}$$
$$\msc \mapsto w^{-1}dw = u^{-1}du + v^{-1}dv.$$

It is possible to obtain similar presentations for
$\wh{\mfsl}_n(B)$ and $\wh{\mfsl}_n(C)$, which are Lie subalgebras of
$\wh{\mfsl}_n(A)$.

\subsection{Derivations in the toroidal case}

The Kac-Moody Lie algebras of affine type associated to the semisimple Lie algebra $\mfg$
are obtained by adding a derivation to the universal central extension of $\mfg\ot_{\C}
\C[t^{\pm 1}]$. It has been known to people who work on extended affine Lie algebras how
to extend $\wh{\mfsl}_n(\C[s^{\pm 1}, t^{\pm 1}])$ by adding derivations. We interpret
this in the context of $ \mfsl_n(A) ,\mfsl_m(B), \mfsl_n(C)$, following
the monograph \cite{AABGP}.

To $ \wh{\mfsl}_n(A) $, we add two derivations $\msd_u,\msd_w^v$ which satisfy the commutation relations \[  [\msd_u, m \ot u^{k}w^l \ga] = 0 \text{ if $k \not\equiv 0 \; \mathrm{mod} \, d$}, \;\; [ \msd_u, m \ot u^{dk}v^j \ga ] = k m \ot u^{dk}w^j \ga   \] and $ [ \msd_w^v, m \ot u^{k} w^l \ga] = l m \ot u^{k} w^{l} \ga $. We observe that $ [ \msd_w^v, m \ot u^{k} v^l \ga] = l m \ot u^{k} v^{l}\ga$ and $ [\msd_u, m \ot u^{k} v^l \ga] = \delta_{k \equiv l} \frac{k-l}{d} m \ot u^{k} v^l \ga $. We could define similarly $\msd_v$ and $\msd_w^u$. We then have the relation $\msd_u = -\msd_v$ and $d \cdot \msd_u = \msd_w^u - \msd_w^v$. Dropping the index $w$, we can add two derivations $\msd^u$ and $\msd^v$ to $\wh{\mfsl}_n(A)$.

\begin{definition}
The cyclic double affine Lie algebra $ \ol{\mfsl}_n(A) $ is defined by adding the two derivations $\msd^u$ and $\msd^v$ to $ \wh{\mfsl}_n(A) $. We define similarly $ \ol{\mfsl}_n(B) $ and $ \ol{\mfsl}_n(C) $. 
\end{definition}

\subsection{Triangular decompositions}

We have the following triangular decompositions :
\begin{equation} \mfsl_n(A) \cong ( \mfn^-  A ) \oplus (
\mfh A \oplus I [A,A] )  \oplus ( \mfn^+  A  ), \label{td1} \end{equation}

and $\mfsl_n(A)$ is also isomorphic to
\begin{equation}  \label{td2}
\begin{split}
    & \mfsl_n u^{-1}\C[u^{-1},v^{\pm 1}]\rtimes\Ga \oplus \left(
\bigoplus_{\stackrel{i=1}{j\le -1}}^{d-1}
 I (\C[v^{\pm 1}] u^j \xi^i)   \oplus
\bigoplus_{\stackrel{s\leq -1}{r\not\equiv s}} I(\C u^s v^r) \right) \oplus
\mfn^- \C[v^{\pm 1}]\rtimes\Ga  \\
    \oplus & \left( \mfh \C[v^{\pm 1}]\rtimes \Ga\oplus 
\left( \bigoplus_{1\leq i\leq d-1}
I(\C[v^{\pm 1}]\xi^i) \oplus\bigoplus_{r\not\equiv 0} I(\C v^r) \right) \right) \oplus
\\
    & \mfsl_n u\C[u,v^{\pm 1}]\rtimes\Ga \oplus \left(
\bigoplus_{\stackrel{i=1}{j\ge 1}}^{d-1}
I(\C[v^{\pm 1}] u^j \xi^i)  \oplus \bigoplus_{\stackrel{s\geq 1}{r\not\equiv
s}} I(\C u^s v^r) \right)  \oplus  \mfn^+  \C[v^{\pm 1}]\rtimes\Ga. 
\end{split} 
\end{equation}

and then by, setting $w = uv$, we see that $\mfsl_n(A)$ is moreover
isomorphic to :

\begin{equation}   \label{td3}
\begin{split}
& \mfsl_n u^{-1}\C[u^{-1},w^{\pm 1}]\rtimes\Ga \oplus \left(
\bigoplus_{\stackrel{i=1}{j\le -1,r\in\Z}}^{d-1}
I( \C w^{r} u^j \xi^i)   \oplus
\bigoplus_{\stackrel{s\leq -1}{s\not\equiv 0,r\in\Z}} I(\C w^{r} u^s)  \right)
\oplus \mfn^- \C[w^{\pm 1}]\rtimes\Ga  \\
    \oplus &\left( \mfh \C[w^{\pm 1}]\rtimes \Ga\oplus 
\left( \bigoplus_{1\leq i\leq d-1} I(\C[w^{\pm 1}]\xi^i) \right) \right) \oplus 
\\
    &  \mfsl_n u\C[u,w^{\pm 1}]\rtimes\Ga \oplus  \left(
\bigoplus_{\stackrel{i=1}{j\ge 1,r\in\Z}}^{d-1}
I(\C w^{r} u^j \xi^i)  \oplus \bigoplus_{\stackrel{s\geq 1}{s\not\equiv 0,r\in\Z}}
I(\C w^{r} u^s) \right)  \oplus  \mfn^+  \C[w^{\pm 1}]\rtimes\Ga.  
\end{split}
\end{equation}
In the last two decompositions, one can exchange $u$ and $v$ and get
other decompositions.

The universal central extensions $\wh{\mfsl}_n(A),\wh{\mfsl}_n(B)$ and
$\wh{\mfsl}_n(C)$ also have three triangular decompositions: they are
obtained by adding the center to the middle part.

It is worth looking quickly at $\mfh A \oplus I [A,A]$. We know that $A
\cong M_d ( \C[s^{\pm 1}, t^{\pm 1}] )$, so, if $d>1$, $[A,A] = A$ and $\mfh A \oplus I
[A,A] \cong \mfgl_d(\C[s^{\pm 1}, t^{\pm 1}])^{\oplus n}$.

\section{Representations of cyclic double affine Lie algebras}\label{repcdala}

In this section, we begin to study the representation theory of the algebras 
defined in the previous sections.

\subsection{Integrable and highest weight modules for cyclic double
affine Lie algebras}\label{inthwmod}

At the end of section \ref{cycDALA}, we gave three triangular
decompositions of $\wh{\mfsl}_n(A)$. The first one is analogous to the
one used for Weyl modules in \cite{FeLo}, but in our situation the
middle Lie algebra is not commutative if $\Ga\neq\{ \mathrm{id} \}$. (Note
that an analog of the first triangular decomposition is not known for
quantum toroidal algebras.) We will return to Weyl modules in section \ref{Weylmod} below.

The second triangular decomposition corresponds to the first
triangular decomposition of the cyclic affine Lie algebras for the
parameter $v$, and to the second triangular decomposition of the cyclic
affine Lie algebra for the parameter $u$. Thus it is analogous to the
triangular decomposition used in \cite{Mi1, Naams, He1} to construct
$l$-highest weight representations of quantum toroidal algebras. Again, in our situation, the middle Lie
algebra is not commutative if $\Ga\neq\{ \mathrm{id} \}$.

However, the
middle term $\mfH$ of the last triangular decomposition is a
commutative Lie algebra. Actually, this last case is obtained by
considering the first
triangular decomposition of the cyclic affine Lie algebras for the
parameter $w$ and the second triangular decomposition of the cyclic
affine Lie algebra for the parameter $u$. Under the isomorphism
between $\wh{\mfsl}_n(A)$ and the toroidal Lie algebra
$\wh{\mfsl}_{nd}(\C[u^{\pm 1}, w^{\pm 1}])$ given in proposition
\ref{isotor}, it corresponds to the standard decomposition of
$\wh{\mfsl}_{nd}(\C[u^{\pm 1}, w^{\pm 1}])$ as used in \cite{cl} for a certain central
extension of $\mfsl_{nd}(\C[u^{\pm 1}, w^{\pm 1}])$. (It is analogous to the
decomposition considered
in the quantum case \cite{Mi1, Naams, He1}). In particular, the notions of integrable
and highest weight modules for
this decomposition have been studied
in \cite{cl, Ra1, yy}; we simply reformulate their results for the benefit of the reader
in the next subsection. Integrable highest
weight representations are classified by tuples of $nd-1$ polynomials. In opposition to
the quantum case, evaluation morphisms are available and provide a direct way to construct integrable
representations.

\subsubsection{The standard highest weight structure on $\wh{\mfsl}_n(A)$
and $\wh{\mfsl}_n(B)$}\label{hws}

In this subsection, we include previously known results about the
standard highest weight structure on $\wh{\mfsl}_n(A)$ and on
$\wh{\mfsl}_n(B)$. (Actually, the results below have been proved only
for $\wh{\mfsl}_n(A)$, but the proofs are similar for
$\wh{\mfsl}_n(B)$.) Let $\mfg^{\pm}$ be the positive and negative parts of the triangular
decomposition \eqref{td3}, and let $\mfH$ be the middle part.

Instead of highest weight vectors, we have to consider the notion of pseudo-highest
weight vectors. Suppose that $\La =( \la_{i,j,r} )_{0\le i\le n-1,r \in \Z}^{0\le j\le d-1}$ 
where $\la_{i,j,r} \in\C$. We define the Verma module $M(\La)$ to be the $\wh{\mfsl}_{n}(A)$-module
induced from the $\mfH\oplus\mfg^+$ representation generated by the vector $v_{\La}$ on
which $\mfg^+$ and $\msc$ act by zero and $H_{i,j,r}$ acts by multiplication by $\la_{i,j,r}$  (in
the notation of proposition \ref{DALApres}). We have a grading on the Verma module and so
it has a unique simple quotient $L(\lambda)$.  For $\mu\in\mfH^*$, we define the notion
of weight space $V_\mu$ of a representation $V$ as usual.

\begin{definition}\label{intdef1}
A module $M$ over $\wh{\mfsl}_n(A)$ is called integrable if $M = \oplus_{\mu\in\mfH^*}
M_{\mu}$ and if the vectors $E_{ij}u^sv^r\ga$ act locally nilpotently if $1\le i\neq
j\le n, r,s\in\Z,\ga\in\Ga$.
\end{definition}

Note that, in the quantum case, a stronger notion of integrability is used
instead of local nilpotency of the operators \cite{cl, Ra1}.

\begin{proposition}\label{irrepAB}\cite{cl}
The irreducible module $L(\La)$ is integrable if and only if, for any $0\le i\le n-1,
0\le j\le d-1$, we have $\la_{i,j,0} \in \Z_{\ge 0}$ and there exist monic polynomials 
$P_{i,j}(z)$ with non-zero constant terms and of degree $\la_{i,j,0}$ such that $ \sum_{r\ge 1}\la_{i,j,r}z^{r-1} =
- \frac{P'_{i,j}(z)}{P_{i,j}(z)} $ and  $ \sum_{r\ge 1}\la_{i,j,-r}z^{r-1} =  - \lambda_{i,j,0}z^{-1} + z^{-2}
\frac{P'_{i,j}(z^{-1})}{P_{i,j}(z^{-1})}  $ as formal power series.
\end{proposition}

The proof of the necessary condition in this proposition reduces to the case of the loop Lie algebra of $\mfsl_2$, which is why it extends automatically from the affine to the double affine setup.
The sufficiency is proved as in \cite{cl} using tensor products of evaluation modules (the formulas for the power series
are a bit different from those in [Proposition 3.1, \cite{cl}] as we use 
different variables; see also \cite{Mi2}). In the quantum context, the polynomials $P_{i,j}(z)$ are called Drinfeld polynomials.  A similar criterion for integrability exists for quantum toroidal algebras: this is
explained in \cite{He1} after proving that certain subalgebras of a quantum toroidal
algebras are isomorphic to the quantized enveloping algebra of $\mfsl_2(\C[t^{\pm
1}])$. Affine Yangians for $\mfsl_n$ are built from copies of $Y(\mfsl_2)$, the Yangian
for $\mfsl_2$: this follows from the PBW property of affine Yangians proved in
\cite{Gu2}; therefore, a similar integrability condition holds for them also.

\subsubsection{The standard highest weight structure on $\wh{\mfsl}_n(C)$}

  The case which interests us more and presents some novelty is
$\wh{\mfsl}_n(C)$, because the Lie algebra $\mfsl_n(C)$ is not isomorphic to
$\mfsl_{nd}(\C[s,t])$.  The three triangular decompositions of $\wh{\mfsl}_n(A)$ given at
the end of the section \ref{cycDALA} yield such
decompositions for $\wh{\mfsl}_n(C)$ and we consider the one coming from
\eqref{td3}. More precisely, $\mfsl_n(C)$ can be decomposed into the direct sum of the
following three subalgebras.
\begin{equation} \label{td3C}
\begin{split}
& \left( \mfsl_n v\C[v,w]\rtimes\Ga \oplus  \left(
\bigoplus_{j\ge 1, 1\leq i\le d-1} I(\C[w] v^j \xi^i)   \oplus
\bigoplus_{r\ge 1,r\not\equiv 0} I(\C[w] v^r)  \right) \oplus \mfn^- \C[w]\rtimes\Ga
\right)  \\
  \oplus & \left( \mfh \C[w]\rtimes \Ga \oplus 
\Big( \bigoplus_{1\leq i\leq d-1} I(w\C[w]\xi^i) \Big) \right) 
\\
  \oplus &\left( \mfsl_n u\C[u,w]\rtimes\Ga \oplus \left(
\bigoplus_{j\ge 1, 1\leq i\leq d-1}
I(\C[w] u^j \xi^i)  \oplus \bigoplus_{r\geq 1,r\not\equiv 0}
I(\C[w] u^r) \right)  \oplus  \mfn^+  \C[w]\rtimes\Ga
    \right). 
\end{split}
\end{equation}

We have an embedding $\wh{\mfsl}_n(C) \into \wh{\mfsl}_{nd}(\C[u^{\pm 1},v])$, but in
order to classify integrable, highest weight  representations of $\wh{\mfsl}_n(C)$, we
will instead use the following presentation.

\begin{proposition}\label{presC}
The Lie algebra $\wh{\mfsl}_n(C)$ is isomorphic to the Lie algebra $\mfk$ generated by
the elements $X_{i,j,r}^{\pm}, H_{i,j,r}, X_{0,j,r}^+, X_{0,j,r+1}^-, H_{0,j,r+1}$ for
$1\le i\le n-1, 0\le j\le d-1, r\in\Z_{\ge 0}$, satisfying the following list of
relations.
\[ [H_{i_1,j_1,r_1}, H_{i_2,j_2,r_2}] = 0, \;\; [H_{i_1,j_1,0}, X_{i_2,j_2,r_2}^{\pm}] =
\pm c_{f(i_1,j_1),f(i_2,j_2)}  X_{i_2,j_2,r_2}^{\pm}, \]
   \[ [H_{i_1,j_1,r_1+1}, X_{i_2,j_2,r_2}^{\pm}]  =
[H_{i_1,j_1,r_1}, X_{i_2,j_2,r_2+1}^{\pm}],\;\;
[X_{i_1,j_1,r_1+1}^{\pm}, X_{i_2,j_2,r_2}^{\pm}]  =
[X_{i_1,j_1,r_1}^{\pm}, X_{i_2,j_2,r_2+1}^{\pm}], \]
\[ [X_{i_1,j_1,r_1}^+, X_{i_2,j_2,r_2}^{-}] =
\delta_{i_1=i_2}\delta_{j_1=j_2} H_{i_1,j_1,r_1+r_2},\;\;
 ad(X_{i_1,j_1,r_1}^{\pm})^{1-c_{f(i_1,j_1),f(i_2,j_2)}}(X_{i_2,j_2,r_2}^{\pm})
= 0  .\]
\end{proposition}

The proof of proposition \ref{DALApres} in \cite{MRY} works also for $\wh{\mfsl}_n(B)$ and it is possible to deduce from it proposition \ref{presC}. See \cite{Gu3} for more details.

The elements with $i\neq 0$ generate a Lie subalgebra isomorphic to
$\mfsl_n(\C[w]\rtimes\Ga)$, the $X_{i,j,r}^+, 0\le i\le n-1, 0\le j \le d-1, r\in\Z_{\ge
0}$ generate the positive part $\wh{\mfsl}_n(C)^+$  of the decomposition \eqref{td3C} and the
$X_{i,j,r}^-, X_{0,j,r+1}^-, 0\le i\le n-1, 0\le j \le d-1, r\in\Z_{\ge 0}$ generate the
negative part $\wh{\mfsl}_n(C)^-$. Note that $\wh{\mfsl}_n(C)^+ \cong \wh{\mfsl}_n(C)^-$ via
$X_{i,j,r}^+ \mapsto X_{i,j,r}^-$ for $1\le i\le n-1$ and $X_{0,j,r}^+ \mapsto
X_{0,j,r+1}^-$. The elements with $r=0$ generate a copy of $\mfsl_n(\C[u] \rtimes \Ga)$,
whereas the elements $X_{i,j,0}^{\pm}, X_{0,j,1}^-$ with $1\le i\le n-1, 0\le j\le d-1$
generate a Lie subalgebra isomorphic to $\mfsl_n(\C[v] \rtimes \Ga)$. For a fixed $0\le j\le d-1$, the elements
$X_{0,j,r}^+, X_{0,j,r+1}^-, H_{0,j,r+1}$ for all $r\in\Z_{\ge 0}$ generate a subalgebra of $\mfsl_2(\C[w])$ which, as a vector
space, is $\mfn_2^- w\C[w] \oplus \mfh_2 w\C[w] \oplus \mfn_2^+ \C[w]$ where the index $_2$ indicates the corresponding subalgebra of $\mfsl_2$. Let us denote this subalgebra of $\mfsl_2(\C[w])$ by $\check{\mfsl_2}(\C[w])$.

Integrability of representations of $\wh{\mfsl}_n(C)$ has the same meaning as in
definition \ref{intdef1}. As for $\wh{\mfsl}_n(A)$ and $\wh{\mfsl}_n(B)$, we have Verma
modules $M(\La)$ and their irreducible quotients $L(\La)$ for each pseudo-weight $\La =( \la_{i,j,r} \in\C )$ with $0 \le i\le n-1,r \in \Z, 0\le j\le d-1 $ but $r\ge 1$ if $i=0$; the highest weight cyclic generator is again denoted $v_{\Lambda}$. 

\begin{proposition}\label{irrintC}
The irreducible module $L(\La)$ is integrable if and only if $\la_{i,j,0} \in \Z_{\ge 0}$ for $1\le i\le n-1,
0\le j\le d-1$,  and there exist monic polynomials $P_{i,j}(z)$ for $0\le i\le n-1, 0\le j\le d-1$, such that
$\sum_{r\ge 1}\la_{i,j,r}z^{-r-1} = -\mathrm{deg}(P_{i,j}) z^{-1} + \frac{P'_{i,j}(z)}{P_{i,j}(z)}$ as formal power series and $P_{i,j}(z)$ is of
degree $\la_{i,j,0}$ if $1\le i\le n-1, 0\le j\le d-1$.
\end{proposition}

Let us say a few words about the proof. The main difference with the cases
$\wh{\mfsl}_n(A)$ and $\wh{\mfsl}_n(B)$ is that we do not have the elements $X_{0,j,0}^-, \, 0\le j\le d-1$.
However, it is still possible to apply the same argument as in the affine $\mfsl_2$-case
\cite{ca}, modulo some small differences. For instance, proposition 1.1 in \cite{ca} is
fundamental for the rest of that article, but it cannot be applied in our case when $i=0$: what we need instead is an expression for $(X_{0,j,0})^r(X_{0,j,1}^{-})^r$.  Proposition 1.1 in \cite{ca} is a consequence of lemma 7.5 in \cite{Ga}. To obtain an expression for $(X_{0,j,0})^r(X_{0,j,1}^{-})^r$, we just have to apply the automorphism of $\check{\mfsl_2}(\C[w])$ given by $ E_{21}w^{r+1} \mapsto E_{12}w^r, \, E_{12}w^{r} \mapsto E_{21}t^{r+1}$ and $ (E_{11}-E_{22})t^{r+1} \mapsto (E_{22}-E_{11})t^{r+1} $ for $r\in\Z_{\ge 0}$.  We lose the condition that the degree of
$P_{0,j}(z)$ is $\la_{0,j}$ because $\wh{\mfsl}_n(C)$ does not contain the $\mfsl_2$-copies
generated by $X_{0,j,0}^{\pm}, H_{0,j,0}$. The degree of $P_{0,j}(z)$ is the smallest
integer $r$ such that $(X_{0,j,1}^-)^{r+1} v_{\Lambda}=0$. The proof of the sufficiency
of the condition in the proposition consists in the construction of an integrable
quotient of the Verma module $M(\Lambda)$ using tensor products of evaluation modules, as
in \cite{cl}. Note that also in the case of $\wh{\mfsl}_n(C)$, the degree of $P_{0,j}(z)$ is the smallest integer $r$ such that $(X_{0,j,1}^-)^{r+1}$ acts by zero on the cyclic highest weight vector.

\subsection{Weyl modules for $\mfsl_n(\C[u,v] \rtimes \Gamma)$}\label{Weylmod}

For a Lie algebra $\mfg$ and a commutative $\C$-algebra $\mcA$, we will often use the notation $\mfg(\mcA)$ or $\mfg\mcA$ to denote $\mfg \ot_{\C} \mcA$. If $\mathfrak{g} = \mathfrak{n}^+ \oplus
\mathfrak{H}\oplus \mathfrak{n}^-$  the Weyl modules \cite{FeLo} are certain representations of
$\mathfrak{g}(\mathcal{A})$ generated by a weight vector $v$
satisfying
$\big(\mathfrak{n}^+ (\mathcal{A})\big).v = 0$. (In this subsection, we consider only the
local Weyl modules, not the global ones.)  The motivation to study Weyl modules is
that they should be simpler to understand than the finite dimensional irreducible
modules. 
This is what happens in the quantum affine setup
where the Weyl modules for the affine Lie  
algebras are closely related to finite dimensional irreducible modules of the  
corresponding affine quantum group when $q\mapsto 1$ \cite{ChPr, ChLo}.
The definition of Weyl modules depends on the choice of a triangular
decomposition, but only the first of our triangular decompositions for cyclic double
affine Lie algebras seems appropriate.
It should be noted that we cannot use proposition \ref{isotor} to deduce results about
Weyl
modules for $\mfsl_n(A)$ in our context because, when $\Ga$ is non-trivial, the
isomorphism in that proposition does not map the triangular decomposition \eqref{td1}
of $\mfsl_n(A)$ to the decomposition of $\mfsl_{nd}(\C[s^{\pm 1},
t^{\pm 1}])$ considered in \cite{FeLo}, which is $\mfsl_{nd}(\C[s^{\pm
1}, t^{\pm 1}]) =\mfn_{nd}^-\C[s^{\pm 1}, t^{\pm
1}]\oplus\mfh_{nd}\C[s^{\pm 1}, t^{\pm 1}]\oplus\mfn_{nd}^+\C[s^{\pm
1}, t^{\pm 1}]$.

In this subsection, we need a stronger definition of integrability than the one presented
in \ref{intdef1}. It is the same as the one used in \cite{FeLo}.

\begin{definition}\label{intdef2}
A module $M$ over a Lie algebra of the type $\mfg \ot_{\C} \mcA$ is said to be integrable
if $M_{\mu}$ is non-zero for only finitely many $\mu \in P$.
\end{definition}

When $\mcA$ is the coordinate ring of an affine algebraic variety $X$, Weyl modules are
associated to multisets of
points of $X$. In the simplest case of a (closed) point, we have an augmentation $\mcA
\lra \C$. However, when it comes to the triangular decomposition \eqref{td1}, the middle
term is isomorphic to $\mfd \ot_{\C} M_d(\C[s^{\pm 1}, t^{\pm 1}])$, where $\mfd$ is the
abelian Lie algebra of the diagonal matrices in $\mfgl_n$. When $d=1$, we are exactly in
the same situation as in \cite{FeLo} (with $X$ the two-dimensional torus $\C^{\times}
\times \C^{\times}$), but when $d>1$, the Lie algebra is non-commutative. On new possibility is to consider maximal two-sided ideals in $A \rtimes \Ga$, or, equivalently, augmentation
maps. We are thus led to the following definition, which we can formulate in a more
general setting. 

\begin{definition}\label{lWeylmod} Let $\mcA$ be a commutative, finitely generated algebra
with a unit and let $G$ be a finite group acting on $\mcA$ by algebra automorphisms.
Consider an augmentation  $\epsilon$ of $\mcA\rtimes G$, that is, an algebra
homomorphism $ \mcA\rtimes G \lra \C$, and let $\la \in \mfh^*$ be a dominant integral 
weight. We define the Weyl module $W_{\mcA \rtimes G}^{\epsilon}$ to be the maximal
integrable cyclic $\mfsl_n(\mcA\rtimes G)$-module generated by a vector $v_{\la}$
such that, for $a \in \mcA\rtimes G $: \begin{equation*} (ha)(v_{\la}) = \la(h)
\epsilon(a) v_{\la}, \;\; \mfn^+ ( \mcA\rtimes G)(v_{\la}) = 0.\end{equation*}
\end{definition}

The existence of such a maximal module can be proved as in \cite{FeLo} using the notion
of global Weyl module. This definition agrees with the one used in the paper \cite{FeLo}
in the case $G=\{\mathrm{id}\}$. We note that, by $\mfsl_2$-theory, we have that
$f_i^{\la(h_i)+1} v_{\la}=0$ where, as usual, we denote by $f_i,h_i,e_i, 1\le i \le n-1$
the standard Chevalley generators of $\mfsl_n$.

It turns out that Weyl modules for the smash product $\mcA \rtimes G$ are related to
Weyl modules for a much smaller ring.  We have a decomposition of $G$-modules $\mcA =
\mcA^{G} \oplus \mcA'$, where $\mcA'$ is the subrepresentation without invariants.
Let us denote by $\ol{\mcA}$ the quotient of $\mcA$ by the two-sided ideal generated by
$\mcA'$.  Note that it may be much smaller than $\mcA^{G}$, it can even reduce to
$\C$, for instance, when $G$ is $\Z/d\Z$ and $\mcA = \C[u]$ or even $\mcA = \C[u,v]$.

Consider an augmentation  $\epsilon$ of $ \mcA \rtimes G $.  Note that $\mcA'
\subset [\mcA\rtimes G ,\mcA\rtimes G ]$, so $\epsilon(\mcA')=0$ and $\epsilon$
descends to an augmentation $\overline{\epsilon}$ of $\ol{\mcA}$.

\begin{theorem}\label{isoWeyl} Let $\la\in\mfh^*$ be a dominant integral weight. We have an
isomorphism of modules over $\mfsl_n(\mcA \rtimes G)$: \[ W_{\mcA\rtimes
\C[G]}^\epsilon(\lambda) \cong W_{\ol{\mcA}}^{\overline{\epsilon}} (\lambda). \]
\end{theorem}

We have a surjective map $ \mcA \rtimes G \to \ol{\mcA}$, hence $\mfsl_n(\mcA
\rtimes G)$ acts on $W_{\ol{\mcA}}^{\overline{\epsilon}} (\lambda)$ and this yields
a surjective map $W_{\mcA \rtimes G}^\epsilon(\lambda) \to
W_{\ol{\mcA}}^{\overline{\epsilon}} (\lambda)$ of modules over $\mfsl_n(\mcA \rtimes
G)$. The ring $\ol{\mcA}$ is the quotient of $\mcA \rtimes G$ by the ideal
$J_\epsilon$ generated by $\mcA'$ and
elements $\gamma -\epsilon(\gamma)$ in $\C[G]$ for $\gamma \in G$. The only
thing we need to show is that $\mfsl_n \otimes J_\epsilon$ acts on $W_{\mcA \rtimes
G}^\epsilon(\lambda)$ by zero.

Actually, since $\mfsl_n \otimes J_\epsilon$ is an ideal, it is enough to show that it
acts by zero on the highest weight vector $v_{\la}$. As $\epsilon(J_\epsilon) = 0$, this
is true for $\mfb \otimes J_\epsilon$, so it remains to prove our claim for $E_{ij}
\otimes \epsilon$ with $i>j$. Now the question is reduced to $\mfsl_2$-case, so, to simplify
the notation, let us set $f=E_{ij}$, $e = E_{ji}$ and $h = E_{ii}-E_{jj}$.

\begin{lemma}\label{sym}
Let $P \in \mcA \rtimes \Gamma$. If $\epsilon(P)= 0$, then $(f \otimes P^{\la(h)}) v_\lambda = 0$.
\end{lemma}

\begin{proof}
We already know that $f^{\la(h)+1} v_{\lambda}=0$. Applying $e\ot P$ $j$ times and using
the assumption $\epsilon(P)= 0$ yields $(f^{\la(h)+1-j} \ot P^j) v_{\lambda}=0$, so
taking $j=\la(h)$ proves the lemma.
\end{proof}

\begin{lemma}\label{mul}
If $(f \otimes P) v_\lambda = 0$ and $Q\in\mcA$, then $\big( f \otimes (PQ+QP)\big)
v_\lambda = 0$.
Moreover, if $Q$ belongs to the commutator $ [ \mfsl_n(\mcA \rtimes G),\mfsl_n(\mcA
\rtimes G)]$, then
$(f \otimes PQ) v_\lambda = (f \otimes QP) v_\lambda =0$.
\end{lemma}

\begin{proof}
Applying $f\ot P$ to both sides of $(h \ot Q) v_{\la} = \la(h)\epsilon(Q) v_{\la}$ yields
the first equality. If now, say, $Q=Q_1Q_2-Q_2Q_1$, then $ (E_{ii}+E_{jj}) \ot Q = [h\ot
Q_1, h\ot Q_2]$; starting from $( f \otimes P ) v_\lambda = 0$ and applying $E_{ii}\ot Q,
E_{jj} \ot Q$, we obtain the second equality. (Note that $E_{ii} \ot Q$ belongs to
$\mfsl_2( \mcA \rtimes G)$ since, by assumption, $Q \in [ \mfsl_n(\mcA \rtimes
G),\mfsl_n(\mcA \rtimes G) ]$.)
\end{proof}

\begin{proof}[Proof of theorem \ref{isoWeyl}] First let us show  that
\begin{equation}\label{grz}
(f \otimes (x - \epsilon(x))) v_\lambda =0 \quad \mbox{for}\  x \in \C[G].
\end{equation} Note that $\C[G] = \C c_\epsilon \oplus I_\epsilon$, where
$c_\epsilon^2 = c_\epsilon$,
$\epsilon(c_\epsilon)=1$, and $I_\epsilon$ is the kernel of $\epsilon_{| \C[G]}$.
We have $I_\epsilon^{l+1} = I_\epsilon$, so lemma \ref{sym} yields equation \eqref{grz}
for $x \in I_\epsilon$. Moreover, note that $c_\epsilon -1$ belongs to $I_\epsilon$, so
we have \eqref{grz} also for $x = c_\epsilon$ and, therefore, for any $x$.

Since $\mcA' \in [ \mfsl_n(\mcA \rtimes G),\mfsl_n(\mcA \rtimes G) ]$ and
$\mcA^G$ commutes with $\C[G]$, lemma \ref{mul} implies that $(f \otimes \mcA(x -
\epsilon(x))) v_\lambda =0$.

It remains to show that $(f \otimes \mcA')v_\lambda = 0$.  By lemma \ref{mul}, for any $a
\in \mcA'$, $\gamma \in G$, we have \[ (f\otimes \gamma a) v_\lambda = (f\otimes
a\gamma ) v_\lambda = (f\otimes a \epsilon(\gamma))v_\lambda. \]
So \[  \big( f\otimes (a - \gamma^{-1}(a))\big)v_\lambda = \frac{1}{\epsilon(\gamma)}
(f\otimes \big( \gamma a - a \gamma)\big) v_\lambda = 0. \] At last, note that the
elements $(a - \gamma^{-1}(a))$ for $a \in \mcA'$, $\gamma \in G$, span $\mcA'$.
\end{proof}

\subsection{Weyl modules associated to rings of invariants}\label{Weylmodinv}

When $\mcA$ is a commutative, unital, finitely generated $\C$-algebra, Weyl modules for $\mfsl_n\ot_{\C} \mcA$ can be attached to multisets of points in $\mathrm{Spec}(\mcA)$ or, more generally, to any ideal in $\mcA$. In this subsection, we first apply an approach due to Feigin-Loktev \cite{FeLo} and Chari-Pressley \cite{ChPr} to describe certain local Weyl modules for $\mfsl_n(\C[u,v])$ and $\mfsl_n(\C[u,v]^{\Gamma})$: this approach realizes them as the Schur-Weyl duals of certain modules of co-invariants.  A natural question to ask is: what is the dimension of these local Weyl modules? For loop algebras $\mfsl_n(\C[u,u^{-1}])$, this question was fully answered in \cite{ChLo}. For $\mfsl_n(\C[u,v])$ and a multiple of the fundamental weight of the natural representation of $\mfsl_n$ on $\C^n$, this problem was solved in \cite{FeLo}, but the answer relies on the difficult theorem of M. Haiman on the dimension of diagonal harmonics \cite{Ha2}. To compute the dimension of the Weyl modules that we introduce below, we would need an extension of Haiman's theorem to certain rings of coinvariants attached to wreath products of the cyclic group $\Z/d\Z$, but this is still an open problem as far as we know.  At least, we are able to provide a lower bound for the dimension of some local Weyl modules by using a theorem of R. Vale \cite{Va}, which generalizes an earlier result of I. Gordon \cite{Go1}.

\begin{definition}
Let $U$ be a representation of $\mfsl_n \ot \mcA$ and $\mu \in \mfh^*$. Suppose that we have an augmentation map $\epsilon: \mcA \lra \C$. A vector $v_{\mu} \in U$ is called a highest weight vector if $(g\ot a) v_{\mu} =0$ when $g\in\mfn^+, a\in\mcA$ and $(h\ot a) v_{\mu} = \mu(h) \epsilon(a) v_{\mu}$ for all $h\in\mfh, a\in \mcA$.
\end{definition}

\begin{theorem}\cite{FeLo}
Suppose that $\mu \in \mfh^*$ is a dominant integral weight. There exists a universal finite dimensional module $W_{\eps}^{\mcA}(\mu)$ such that any finite dimensional representation of $\mfsl_n \ot \mcA$ generated by a highest weight vector $v_{\mu}$ is a quotient of $W_{\eps}^{\mcA}(\mu)$. 
\end{theorem}

Let us recall a general result of \cite{FeLo}, see also \cite{Lo}.  The symmetric group $S_l$ acts on $\mcA^{\ot l}$ and we can form the quotient $\mathsf{DH}_l(\mcA) = \mcA^{\ot l} / ( \mathrm{Sym}^l(\mcA)_{\eps})$ of $\mcA^{\ot l}$ by the ideal generated by the tensors invariant under the action of $S_l$ and which are in the kernel of $\eps$ (extended as an augmentation $\mathrm{Sym}(\mcA)$). When $\mcA = \C[u,v]$, this quotient is called the space of diagonal co-invariants.  

If $E$ is a representation of $S_l$, denote by $\mathsf{SW}_l^n(E)$ the representation of $\mfsl_n$ given by $\big( (\C^n)^{\ot l} \ot E \big)^{S_l}$. This is the classical Schur-Weyl construction. Note that $ \mathsf{DH}_l(\mcA) $ is a representation of $S_l$. More generally, as is observed in \cite{Lo}, if $E$ is a representation of the smash product $(\mcA^{\ot l}) \rtimes S_l $, then $\mathsf{SW}_l^n(E)$ is a representation of $\mfsl_n(\mcA)$. 

Let $\omega_1$ be the fundamental weight of $\mfsl_n$ which is the highest weight of its natural representation $\C^n$.

\begin{theorem}\cite{FeLo}
The Weyl module $W_{\eps}^{\mcA}(l\omega_1)$ of $\mfsl_n(\mcA)$ is isomorphic to $\mathsf{SW}_l^n( \mathsf{DH}_l(\mcA) )$.
\end{theorem}

When $\mcA = \C[u,v]$, the dimension of the ring of diagonal harmonics $ \mathsf{DH}_l(\mcA) $ is a difficult result proved by M. Haiman in \cite{Ha2}. Until after proposition \ref{SWdhc}, we will assume that $\Ga$ is an arbitrary finite subgroup of $SL_2(\C)$.  The preceding theorem can be applied to $\mcA = \C[u,v]^{G} $ and $\eps: \C[u,v]^{G} \lra \C$ the homomorphism given by the maximal ideal $\C[u,v]^{G}_{+}$ corresponding to the singularity. (Here, $G$ is an arbitrary finite subgroup of $SL_2(\C)$.)  It gives a nice description of the Weyl module for multiples of $\omega_1$, but, to compute its dimension, we would have to know more about the structure of $\mathsf{DH}_l( \C[u,v]^{G} )$ as a module for $S_l$: as far as we know, this is still an open problem when $G \neq \{ 1 \}$.  We are, however, able to obtain a partial result by 
considering the ring $\mcA = \C[u,v]$ but with highest weight conditions on $\mfh \otimes \C[u,v]^{\Ga}$, and from it we can deduce a lower bound when $\mcA=\C[u,v]^{\Ga}$.

By $\dhc{G}{l}$, denote the quotient of 
$\C[u_1, \dots, u_l, v_1, \dots, v_l]$ by the ideal generated by
$S_l\ltimes G^{\times l}$-invariants with zero at the origin. 
This is a module for $S_l \ltimes G^{\times l}$,
and, moreover, for $(S_l \ltimes G^{\times l})\ltimes \C [u,v]^{\otimes l}$.

\begin{definition}\label{defWeyl}
Let $\mu \in \mfh^*$ be a dominant integral weight and let $\nwmod{G}(\mu)$ be the maximal finite dimensional module over
$\mfsl_n(\C[u,v])$, generated by a vector $v_{\mu}$ such that
$$(\mfn^+ \otimes \C[u,v]) v_{\mu} = 0, \qquad (h \otimes P) v_{\mu} = \mu(h) P(0,0) v_{\mu}  \ \  \mbox{for}\ 
h \in \mfh, \  P \in \C [u,v]^G.$$  
We say that $\nwmod{G}(\mu)$ is the Weyl module for $\mfsl_n(\C[u,v])$ associated to the ideal $(\C [u,v]_+^G)$.
\end{definition}
\begin{remark}
The existence of a maximal finite dimensional module with this property can be proved as in \cite{FeLo}.  
\end{remark}

\begin{proposition}\label{SWdhc}
The Weyl module $\nwmod{G}(l\omega_1)$ is Schur-Weyl dual to $ \dhc{G}{l} $, i.e. $\nwmod{G}(l\omega_1) = \mathsf{SW}_l^n(\dhc{G}{l} )$. 
\end{proposition}

\begin{proof}
The argument is the same as the one used in \cite{FeLo}, so we just sketch it for the benefit of the reader. The Weyl module $\nwmod{G}(l\omega_1)$ is the quotient of the global Weyl module of $\mfsl_n(\C[u,v])$ for the weigh $l\omega_1$ by the submodule generated by $\left(h \otimes \C [u,v]^G_+\right) v_{l\omega_1}$.  The global Weyl module for this weight is $\mathrm{Sym}^l(\C^n \otimes \C [u,v])$ and we have to quotient by the submodule generated by the action of $\mathrm{Sym}^l(\C [u,v]^G_+)$. Thus the Weyl module $\nwmod{G}(l\omega_1)$ is obtained by applying the Schur-Weyl construction to the quotient of $\C[u_1, \dots, u_l, v_1, \dots, v_l]$ by the ideal generated by the $S_l\ltimes G^{\times l}$-invariant polynomials. 
\end{proof}

The following theorem of R. Vale is a generalization of a theorem of I. Gordon for $S_l$ \cite{Go1}.

\begin{theorem}\cite{Va}\label{quotdc}
The representation $\dhc{\Ga}{l}$ has a quotient $\dhc{\Ga}{l}^0$ such that the 
trace on $\dhc{\Ga}{l}^0\otimes {\rm Sign}$ of a permutation $\sigma \in S_l$ consisting of  $s$ cycles is equal to
$(dl+1)^s$.
\end{theorem}

Now let us apply Theorem~\ref{quotdc} to the character calculation for  $\nwmod{\Ga}(l\omega_1)$.
Let $F(l,k)$ be the set of functions from $\{1,\dots, l \}$ to
$\{1, \dots, k\}$. This set admits an action of $S_l$
by permutation of the arguments. Denote by 
$\C F(l,k)$ the corresponding complex representation of 
$S_l$.

\begin{lemma}\label{charFlk}
Suppose that $\sigma \in S_l$ is a product of $s$ cycles.
Then the trace of $\sigma$ on $\C F(l,k)$ is equal to $k^{s}$.
\end{lemma}

\begin{proof}
The trace of $\sigma$ is equal to the number of functions stable under the action of $\sigma$. A function is stable if it has the same value on all the elements of each cycle, so it is determined by $s$ elements of   $\{1, \dots, k\}$. 
\end{proof}

\begin{lemma}\label{SWFlk}
The $\mfsl_n$-module $\mathsf{SW}_l^n( \C F(l,k)\otimes {\rm Sign} )$ is isomorphic to $\bigwedge^l\big( (\C^n)^{\oplus k}\big)$.
\end{lemma}

\begin{proof}
Note that $(\C^n)^{\otimes l} \otimes \C F(l,k)$ is isomorphic to $\big( (\C^n)^{\oplus k} \big)^{\otimes l}$ 
as $SL_n \times S_l$-module. The isomorphism can be constructed as the map sending $(v_1\otimes \dots \otimes v_l) \otimes f$ to $v_1^{(f(1))} \otimes \dots \otimes v_l^{(f(l))}$, where $v^{(i)}$ belongs to the $i$-th summand of $(\C^n)^{\oplus k}$.
Then the lemma follows by restricting  this isomorphism to the ${\rm Sign}$ component.
\end{proof}

\begin{theorem}\label{Wmodquot}
The Weyl module $\nwmod{\Ga}(l\omega_1)$ has a quotient which, as a representation of $\mfsl_n$, is isomorphic to $\bigwedge^l\big( (\C^n)^{\oplus (dl+1)}\big)$.
\end{theorem}

\begin{proof}
Since $\dhc{\Ga}{l}^0$ is a quotient of $\dhc{\Ga}{l}$, we have that $\mathsf{SW}_l^n(\dhc{\Ga}{l}^0)$ is a quotient of $\nwmod{\Ga}(l\omega_1)$ by proposition \ref{SWdhc}. Lemma \ref{charFlk} and theorem \ref{quotdc} imply that $\dhc{\Ga}{l}^0$ is isomorphic to $\C F(l,k)\otimes {\rm Sign}$ with $k=dl+1$. By lemma \ref{SWFlk}, $\mathsf{SW}_l^n(\dhc{\Ga}{l}^0)$ is thus equal to $\bigwedge^l\big( (\C^n)^{\oplus (dl+1)}\big)$.
\end{proof}

\begin{corollary}\label{lbWmod1}
The dimension of $\nwmod{\Ga}(l\omega_1)$ is bounded below by $ \left( \begin{array}{c} n(dl+1) \\ l \end{array} \right) $.
\end{corollary} 

By modifying slightly the argument in the previous paragraphs, we can give a lower bound also for some local Weyl modules when $\mcA=\C[u,v]^{\Ga}$. Let us introduce the following action of $\Ga^{\times l}$ on $ \C F(l,k)$: if we fix a generator $\xi_i$ of the $i^{th}$ copy of $\Ga$ in $\Ga^{\times l}$ and if $f\in \C F(l,k)$, then $\xi_i(f)=\zeta^{j-1} f$ if $f(i)=j$. This action can be combined with the one associated to $S_l$ to obtain an action of $\Ga^{\times l} \rtimes S_l$. The trace of $\xi_i$ on $\C F(l,k)$ is $\left(\sum_{j=0}^{k-1} \zeta^j \right)k^{l-1}$, which equals $\left(\sum_{j=0}^{\ol{k}} \zeta^j \right)k^{l-1}$ where $0\le \ol{k}\le d-1$ and $k-1\equiv\ol{k} \,\mathrm{mod}\,(d)$. In particular, if $k\equiv 1 \,\mathrm{mod}\,(d)$, then this trace is $k^{l-1}$.

Theorem \ref{quotdc} (see \cite{Va}) also states that the trace of $\xi_i$ on $\dhc{\Ga}{l}^0\otimes {\rm Sign}$ is equal to $(dl+1)^{l-1}$. That theorem actually applies to any element in $\Ga^{\times l} \rtimes S_l$, and from it we can deduce that we have an isomorphism of $\Ga^{\times l} \rtimes S_l$-modules between  $ \dhc{\Ga}{l}^0 $ and  $\C F(l,ld+1)\otimes {\rm Sign}$. ($\Ga^{\times l}$ acts trivially on ${\rm Sign}$.) Therefore,  $(\dhc{\Ga}{l}^0)^{\Ga^{\times l}} \cong \C F(l,l+1)\otimes {\rm Sign}$ since the functions in $\C F(l,ld+1)$ which are invariants under $\Ga^{\times l}$ can be identified with $\C F(l,l+1)$.

Now, we can repeat the argument we used above. We have that $\mathsf{SW}_l^n((\dhc{\Ga}{l}^0)^{\Ga^{\times l}})$ is a quotient of $\mathsf{SW}_l^n(\dhc{\Ga}{l}^{\Ga^{\times l}})$ isomorphic to $\bigwedge^l\big( (\C^n)^{\oplus (l+1)}\big)$. Let $\nwmods{\Ga}(\mu)$ be defined as $\nwmod{\Ga}(\mu)$ in definition \ref{defWeyl}, but with $\C[u,v]$ replaced by $\C[u,v]^{\Ga}$. The Weyl module $\nwmods{\Ga}(l\omega_1)$ is isomorphic to $\mathsf{SW}_l^n(\mathsf{DH}_l( \C[u,v]^{\Ga}))$.  Since $ \mathsf{DH}_l( \C[u,v]^{\Ga} ) \cong (\dhc{\Ga}{l})^{\Ga^{\times l}}$, we conclude that  $\nwmods{\Ga}(l\omega_1)$ has a quotient isomorphic to $\bigwedge^l\big( (\C^n)^{\oplus (l+1)}\big)$ as $\mfsl_n$-module, whence the following corollary.

\begin{corollary}\label{lbWmod2}
The dimension of $\nwmods{\Ga}(l\omega_1)$ is bounded below by $ \left( \begin{array}{c} n(l+1) \\ l \end{array} \right) $.
\end{corollary}

There is no reason to expect that the lower bound in corollary \ref{lbWmod1} is the best possible. Indeed, we can show that, when $d=2,l=4$, it is too low, following ideas of I. Gordon. In this case, $\Ga^{\times l} \rtimes S_l$ is isomorphic to the Weyl group $W$ of type $B_4$. In \cite{Ha1}, M. Haiman explains that the ring of diagonal coinvariants in this case has dimension $9^4+1$, which is one more than the dimension of a certain quotient introduced in conjectures 7.1.2,7.1.3 and 7.2.3 in \textit{op. cit.} These conjectures were proved by I. Gordon in \cite{Go1} and we denoted above this quotient by $\dhc{\Ga}{l}^0$. 

This means that, in  $\dhc{\Z/2\Z}{4}$, there is a one-dimensional subspace $E$ which carries a non-trivial representation of $W$. There is an action of $\mfsl_2(\C)$ on $\dhc{\Z/2\Z}{4}$ commuting with the action of $S_4$ (this is actually true in $\C[u_1,v_1,\ldots,u_l,v_l]$ for any $l\in\Z_{\ge 1}$), so $E$ is also a representation of $\mfsl_2(\C)$ and must thus be trivial. The standard diagonal element $h \in\mfsl_2(\C)$ acts by $\sum_{i=1}^4 \left(u_i\frac{d}{du_i} - v_i\frac{d}{dv_i} \right)$, so this operators acts trivially on $E$, which implies that the monomials which appear in $E$ have their $u$-degree equal to their $v$-degree. 

The Weyl module $\nwmod{\Z/2\Z}(4\omega_1)$ is obtained by applying the Schur-Weyl functor to $\dhc{\Z/2\Z}{4}$, so, if $n\ge 4$,  $\nwmod{\Z/2\Z}(4\omega_1)$ has dimension greater than $ \left( \begin{array}{c} 9n \\ 4 \end{array} \right) $, which shows that the lower bound in corollary \ref{lbWmod1} is too low.

\section{Matrix Lie algebras over rational Cherednik algebras of rank one}\label{matdiff}

The polynomial ring $\C[u,v]$ can be deformed into the first Weyl algebra $A_1 =
\C\langle u,v \rangle / ( vu - uv - 1) $, which can be viewed as the ring $\mcD(\C)$ of
algebraic differential operators on the affine line $\mathbb{A}_{\C}^1$. Such differential
operators play an important role in the representation theory of Cherednik algebras and
the $G$-DDCA
of \cite{Gu1,Gu2,Gu3} are also deformations of the enveloping algebra of $\mfgl_n(A_1
\rtimes \Ga)$ when $G$ is a finite subgroup of $SL_2(\C)$.

More generally, the rational Cherednik algebra $\msH_{t,\mbc}(G_l)$ for the wreath
product $G_l = G^{\times l} \rtimes S_l$ admits two specializations of particular
interest: $\msH_{t=0,\mbc = \mathbf{0}}(G_l) \cong \C[x_1,y_1, \ldots, x_l,y_l] \rtimes
G_l $ and $\msH_{t=1,\mbc=\mathbf{0}}(G_l) \cong A_l \rtimes G_l$ where $A_l$ is
the $l^{th}$ Weyl algebra. The representation theories of these two algebras differ
greatly: for instance, in the first case,  $\msH_{t=0,\mbc = \mathbf{0}}(G_l)$ has
infinitely many irreducible finite dimensional representations, whereas
$\msH_{t=1,\mbc=\mathbf{0}}(G_l)$ has none. Actually, $\msH_{t=1,\mbc}(G_l)$ does not
have any finite dimensional representations for generic values of $\mbc$. The $\Ga$-DDCA
also admit two such specializations and it is reasonable to expect that their
representation theories will thus differ noticeably. In this article, we want to start
investigating the categories of modules for these two specializations, so, in this section
we will study matrix Lie algebras over rational Cherednik algebras of rank one with $t\neq 0$.

\begin{definition}\label{ratCher}
Let $\mbc = (c_1, \ldots, c_{d-1}) \in \C^{d-1}$. The rational Cherednik algebra
$\msH_{t,\mbc}(\Ga)$ of rank one is the algebra generated by elements $u,v,\ga$ with $\ga
\in \Ga = \Z / d\Z$ and the following relations: \begin{equation*}
\ga u \ga^{-1} = \zeta u, \;\; \ga v \ga^{-1} = \zeta^{-1} v
\end{equation*}
\begin{equation}
vu - uv = t + \sum_{i=1}^{d-1} c_i \xi^i \mbox{ where } \xi \mbox{ is a generator of }
\Ga. \label{rc1}
\end{equation}
\end{definition}

It will also be convenient to rewrite equation \eqref{rc1} in the form $vu - uv = t +
\sum_{i=0}^{d-1} \wt{c}_i (\mbe_i - \mbe_{i+1} ) $ for some $\wt{c}_i\in\C$. We will need
to use later the element $\omega$ which can be written in the following three equivalent
ways: \begin{equation*} \begin{split} \omega  =  & -uv + \sum_{i=0}^{d-1} \wt{c}_i
\mbe_{i+1} =  -vu + t + \sum_{i=0}^{d-1} \wt{c}_{i} \mbe_i     
  =  -\frac{uv+vu}{2} + \frac{t}{2} + \frac{1}{2}\sum_{i=0}^{d-1} \wt{c}_i (\mbe_i +
\mbe_{i+1} ).
\end{split}
\end{equation*}
Then one can check that $[\omega,u]=-tu, \; [\omega,v]=tv$.

\begin{definition}\label{trigCher}
Let $\mbc = (c_1, \ldots, c_{d-1}) \in \C^{d-1}$. We will call trigonometric Cherednik
algebra of rank one the algebra $\mbH_{t,\mbc}(\Ga) = \C[u^{\pm 1}] \ot_{\C[u]}
\msH_{t,\mbc}(\Ga)$.
\end{definition}
\begin{remark}
Trigonometric Cherednik algebras exist only for Weyl groups (real Coxeter groups), but we
propose to use the terminology in the previous definition because it is convenient.
Moreover, as is explained in \cite{Gu3}, $\mbH_{t,\mbc}(\Ga)$ depends actually only on
the $t$ parameter, that is, $ \mbH_{t,\mbc}(\Ga) \cong \mbH_{t,\mbc = \mathbf{0}}(\Ga) $,
  but this is not true for $\msH_{t,\mbc}(\Ga)$.  An explicit isomorphism
$\mbH_{t,\mbc}(\Ga) \iso \mbH_{t,\mbc = \mathbf{0}}(\Ga) $ is given by \[ v \mapsto v +
\left( \sum_{i=1}^{d-1} \frac{c_i}{1-\zeta^{-i}} \xi^i \right) u^{-1}. \] One can also
note that $ \mbH_{t,\mbc}(\Ga)$ is generated by $\omega, u,u^{-1},\Ga$ and that
$[\omega,u^{-1}]= tu^{-1}$.
\end{remark}

The associative algebras $ \mbH_{t,\mbc}(\Ga)$ and $\msH_{t,\mbc}(\Ga)$ can be
turned into Lie algebras in the usual way and the representation theory of a central
extension of the Lie algebra $\mfgl_n(\mbA_1)$ was studied in \cite{BKLY,KaRa}. (Here, $\mbA_1$ is the algebra of differential operators on $\C^{\times}$.)
The case of the quantum torus $ \mcD_q(\C^{\times}) = \C\langle u^{\pm 1},v^{\pm 1}
\rangle / ( vu = q
uv) $ was considered in \cite{BoLi}. In this section, we present some
results about the structure of the Lie algebras $ \mfsl_n(\msH_{t,\mbc}(\Ga)), \mfsl_n(
\mbH_{t,\mbc}(\Ga) ) $ mostly when $t\neq 0$.

The (Lie) algebras $ \mbH_{t,\mbc}(\Ga) $ and $ \msH_{t,\mbc}(\Ga) $  are graded:
$\mathrm{deg}(u)=-1, \mathrm{deg}(v)= 1, \mathrm{deg}(\ga)=0$. This induces gradings on
the associative algebras $M_n( \mbH_{t,\mbc}(\Ga) ), M_n(\msH_{t,\mbc}(\Ga))$ and on the
Lie algebras  $ \mfgl_n(\mbH_{t,\mbc}(\Ga)) ,
\mfgl_n(\msH_{t,\mbc}(\Ga))$. However, we will consider instead the following grading: \[
\mathrm{deg}(E_{ij}v^r u^s \ga) =
(r-s)n+j-i  .\] In the case $\mbH_{t=1,\mbc=\mathbf{0}}(\Ga=\{ 1 \})$, this is the
opposite of the principal $\Z$-gradation considered in \cite{BKLY}.  The graded pieces of
degree $k$ will be denoted $ \mfgl_n(\mbH_{t,\mbc}(\Ga))[k]$ and
$\mfgl_n(\msH_{t,\mbc}(\Ga))[k]$.

\subsection{Central extensions}

It was computed in \cite{AFLS} that $HH_1(A_1 \rtimes \Ga) = 0$, hence
$HC_1(A_1 \rtimes \Ga ) = 0$.  Furthermore, it is proved in \cite{EtGi} that, for any $\mbc$, $HH_1( \msH_{t,\mbc}(\Ga) ) =
0$ for all $t\in\C^{\times}$ except in a countable set. For all such values of $t,\mbc$, the Lie algebra $\mfsl_n(\msH_{t,\mbc}(\Ga))$ has no non-trivial central extension. For this reason, contrary to \cite{BKLY}, we will not consider central extensions.

\subsection{Parabolic subalgebras}

In this subsection, we will assume that $t\neq 0$, so, without loss of generality, let us
set $t=1$. For a Lie algebra with triangular decomposition, one usually wants to construct
representations by induction from its non-negative Lie subalgebra (a sort of Borel
subalgebra) or, more generally, from a bigger subalgebra which contains this one.  This
suggests that the following definition may be relevant.

\begin{definition}\label{paralg}\cite{BKLY} A parabolic subalgebra $\mfq$
of the Lie algebra  $\mfgl_n( \msH_{t=1,\mbc}(\Ga) )$ is a graded Lie subalgebra of the form   \[
\mfq = \oplus_{\Z} \mfq[k], \; \mfq[k] = \mfgl_n( \msH_{t=1,\mbc}(\Ga) )[k]
\mbox{ if } k\ge 0, \mfq[k] \subset \mfgl_n( \msH_{t=1,\mbc}(\Ga) )[k] \mbox{ if } k <
0).   \]
\end{definition}

For $k<0$, we can decompose $\mfq[k]$ as \[  \mfq[k] = \bigoplus_{\stackrel{r,l,i,j}{-rn+j-i=k}} E_{ij}  u^r
I_{k}^{i,l} \mbe_l \] for some subspace $I_{k}^{i,l} \subset \C[\omega]$.

\begin{lemma}\label{ideal}
The subspace $I_k^{i,l}$ is an ideal of $\C[\omega ]$.
\end{lemma}

\begin{proof}
Suppose that $ E_{ij} u^r p(\omega )\mbe_l \in \mfq[k]$ with $p(\omega) \in I_k^{i,l}$, and choose $f(\omega ) \in
\C[\omega ]$. If $r=0$, then, since $ E_{ii} f(\omega )\mbe_l \in \mfq[0]$, we deduce
that, if $1\le i\neq j\le n$,  $[E_{ii} f(\omega )\mbe_l ,E_{ij} p(\omega )\mbe_l ]  =
E_{ij} f(\omega )p(\omega ) \mbe_l \in \mfq[k]$, hence $f(\omega )p(\omega ) \in
I_k^{i,l}$.

Now suppose that $r>0$. We want to prove by induction on $a \in \Z_{\ge 0}$ that
$\omega^a p(\omega)\in I_k^{i,l}$. We note that \[ [I \omega^{a+1}, E_{ij}(u^r p(\omega)
\mbe_l) ] = E_{ij}\big( u^r ((\omega-r)^{a+1} - \omega^{a+1}) p(\omega) \mbe_l \big) \in
\mfq[k] \] and that the term of highest power in $(\omega-r)^{a+1} - \omega^{a+1}$ is
$-r(a+1)\omega^a$, so that we can apply induction.
\end{proof}

Following the ideas of \cite{BKLY,KaRa}, we chose a monic generator $b_k^{i,l}(\omega)$
of the principal ideal $I_k^{i,l}$ if this ideal is non-zero; otherwise, we set
$b_k^{i,l}(\omega)=0$. These are called the characteristic polynomials of $\mfq$.

\begin{definition}\label{nondeg}
The parabolic subalgebra $\mfq$ is called non-degenerate if $\mfq[k]\neq 0 \, \forall
k\in\Z$.
\end{definition}

\begin{proposition}\label{nondegp}
A parabolic subalgebra $\mfq$ is non-degenerate if and only if the polynomials
$b_{-1}^{i,l}(\omega), 1\le i\le n, 0\le l\le d-1$ are all non-zero.
\end{proposition}

\begin{proof}
The parabolic subalgebra $\mfq$ is non-degenerate if and only if the polynomials
$b_{k}^{i,l}(\omega), 1\le i\le n, 0\le l\le d-1$ are all non-zero for all $k\in\Z_{\le
-1}$, so it is enough to prove the following for $i=1,\ldots,n$:  if $b_{k}^{i,l}(\omega) \neq 0$ and $b_{-1}^{i-1,l}(\omega) \neq 0$, $b_{k-1}^{i,l}(\omega)$ is also non-zero and divides $b_{k}^{i-1,l}(\omega) b_{-1}^{i,l+r}(\omega-r)$. Here, if $2\le i\le n$, $r$ is determined by $k+i-1=-rn+j$ for some $1\le j\le n$. (We set
$b_{k}^{0,l}(\omega) = b_{k}^{n,l}(\omega)$.)

If $i\neq 1$, we have \begin{equation*}\begin{split} \hspace{1cm} & [ E_{i,i-1} b_{-1}^{i,l+r}(\omega) \mbe_{l+r}, E_{i-1,j} u^r b_{k}^{i-1,l}(\omega) \mbe_l ]  \\ & = E_{ij} u^r b_{-1}^{i,l+r}(\omega-r) b_{k}^{i-1,l}(\omega) \mbe_l - \delta_{j,i} \delta_{r0} E_{i-1,i-1} u^r b_{-1}^{i,l+r}(\omega) b_{k}^{i-1,l}(\omega) \mbe_{l+r}  \in \big[ \mfq[-1],\mfq[k]\big]  \end{split}\end{equation*} and $\big[ \mfq[-1],\mfq[k]\big]
\subset \mfq[k-1]$, so $ b_{-1}^{i,l+r}(\omega-r) b_{k}^{i-1,l}(\omega)  \in I_{k-1}^{i,l} $ and the claim is true if $i\neq n$.

To prove the claim when $i=1$ (and with $r$ determined by $k=-(r+1)n+j$), we consider the following commutator:  \begin{equation*}\begin{split} \hspace{1cm} & [ E_{1n}
u b_{-1}^{1,l+r}(\omega) \mbe_{l+r} , E_{nj} u^r b_{k}^{n,l}(\omega) \mbe_l ] \\
 & = E_{1j} u^{r+1} b_{-1}^{1,l+r}(\omega-r)  b_{k}^{n,l}(\omega) \mbe_l - \delta_{j1} \delta_{r,-1} E_{nn} u^{r+1} b_{k}^{n,l}(\omega-1)b_{-1}^{1,l+r}(\omega) \mbe_{l+r}  \end{split} \end{equation*} which belongs to $ \big[ \mfq[-1],\mfq[k]\big] \subset \mfq[k-1]$, so $  b_{-1}^{1,l+r}(\omega-r)  b_{k}^{n,l}(\omega) \in I_{k-1}^{1,l}$.

Using similar computations, one can prove the following for $i=1,\ldots,n$: if $b_{k-1}^{i,l}(\omega)\neq 0 $, then $b_{k}^{i-1,l}(\omega) \neq 0$ and it divides $b_{k-1}^{i,l}(\omega)$.
\end{proof}

The characteristic polynomials $ b_{-1}^{i,l}( \omega ), 1\le i\le n, 0\le l\le d-1$ can
help us describe the derived Lie subalgebra $[\mfq,\mfq]$.  Set \begin{equation*}
\begin{split}
\mfgl_n(\msH_{t=1,\mbc}(\Ga) )[0,\mbb] =  \mathrm{span} \{ H_i \omega^{r}
b_{-1}^{i+1,l}(\omega) \mbe_{l} | 1\le i\le n-1, r\in\Z_{\ge 0}, 0\le l \le d-1  \} \\
  \oplus \, \mathrm{span}\{ E_{11} uv(\omega+1)^{r} b_{-1}^{1,l}(\omega+1)\mbe_{l+1} - E_{nn}
vu \omega^{r} b_{-1}^{1,l}(\omega)\mbe_{l} | r\in\Z_{\ge 0}    \}
\end{split}
\end{equation*}

\begin{proposition}\label{der}
Let $\mbb=( b_{-1}^{i,l}(\omega))_{1\le i\le n}^{0\le l\le d-1}$ be the first $nd$
characteristic polynomials of the parabolic subalgebra $\mfq$. Then \[ [\mfq,\mfq] =
\left (\bigoplus_{\stackrel{k\in\Z}{k\neq 0}} \mfq[k] \right ) \oplus \mfgl_n(
\msH_{t=1,\mbc}(\Ga) )[0,\mbb]  .\]
\end{proposition}

\begin{proof}
Since $\mfq[k+1] = [\mfq[k],\mfq[1]]$ if $k\in\Z_{\ge 0}$, it is enough to show that
$[\mfq[1],\mfq[-1]] = \mfgl_n( \msH_{t=1,\mbc}(\Ga) )[0,\mbb]$.   \[ [
E_{i,i+1}\omega^{r} \mbe_{l_1} , E_{i+1,i} b_{-1}^{i+1,l_2}(\omega) \mbe_{l_2} ] =
\delta_{l_1l_2} (E_{ii} - E_{i+1,i+1}) \omega^{r} b_{-1}^{i+1,l_2}(\omega) \mbe_{l_1}  \]
\begin{equation*} \begin{split}  [ E_{1n} u b_{-1}^{1,l_1}(\omega) \mbe_{l_1}, E_{n1} v
(\omega+1)^{r} \mbe_{l_2} ] = & \delta_{l_1+1,l_2} \big( E_{11} uv (\omega+1)^{r}
b_{-1}^{1,l_1}(\omega+1)\mbe_{l_2} \\
 & - E_{nn} vu \omega^{r} b_{-1}^{1,l_1}(\omega)\mbe_{l_1} \big) \end{split}
\end{equation*}
\[ [ E_{i,i+1}\omega^{r} \mbe_{l_1} , E_{1n}u b_{-1}^{1,l_2}(\omega) \mbe_{l_2} ] = 0,
\;\; [E_{i+1,i} b_{-1}^{i+1,l_1}(\omega) \mbe_{l_1} , E_{n1} v \omega^r \mbe_{l_2} ] = 0 \text{ if } 1\le i\le n-1\]
\end{proof}

\subsection{Embedding into $\ol{\mfgl}_{\infty}$}

One of the main ingredients in the study of the representation theory of
$\mfgl_n(\mbA_1)$ in \cite{BKLY,KaRa} is an embedding of the algebra $M_n(\mbA_1)$ into
the algebra $\ol{M}_{\infty}$ of infinite matrices with only finitely many non-zero
diagonals. This induces an embedding of the Lie algebra $\mfgl_n(\mbA_1)$ into
$\ol{\mfgl}_{\infty}$. It comes from the action of $\mfgl_n(\mbA_1)$ on $\C^n \otimes \C[u,u^{-1}]$. In this subsection, we obtain similar embeddings for $\mfgl_n(
\msH_{t=1,\mbc}(\Ga) )$ and
$\mfgl_n(\mbA_1 \rtimes \Ga)$ when $\Ga$ is cyclic. The embedding $\mfgl_n(\mbA_1 \rtimes
\Ga) \into \ol{\mfgl}_{\infty}$ is the same as the one considered in
\cite{BKLY,KaRa} when $\Ga$ is trivial, and $ \mfgl_n( \msH_{t=1,\mbc}(\Ga) ) \into \ol{\mfgl}_{\infty}$ comes also from the action of  $\mfgl_n( \msH_{t=1,\mbc}(\Ga) )$ on $\C^n \otimes \C[u,u^{-1}]$ via the Dunkl embedding of $\msH_{t=1,\mbc}(\Ga)$. (We will reserve the notation $M_{\infty}$ and $\mfgl_{\infty}$ for the algebra and the Lie algebra
of infinite matrices with finitely many non-zero entries.)

The space $M_{\infty}$ has a linear basis given by elementary matrices $E_{ij}$ with
$(i,j) \in
\Z \times \Z$. The embedding of associative algebras $\mathbf{\iota}: M_n(\mbA_1
\rtimes\Ga) \into
\ol{M}_{\infty}$ is given explicitly by the formula  \[ \mathbf{\iota} (E_{ij} u^{s}
\omega^r \mbe_k) = \sum_{l \in \Z}
(-ld-k)^r E_{(ld+k+s)n + i - 1, (ld+k)n + j - 1}  , \;\; s\in \Z, 0\le k \le d-1.\] This
restricts to an embedding $\iota: M_n(\msH_{t=1,\mbc}(\Ga)) \into \ol{M}_{\infty}$ by
pulling back via $ \msH_{t=1,\mbc}(\Ga) \into \mbH_{t=1,\mbc}(\Ga)  \cong \mbA_1 \rtimes
\Ga$. Explicitly, since  \[ v = -u^{-1} (\omega - \sum_{l=0}^{d-1} \wt{c}_{l-1} \mbe_{l}
), \] we get  \[ \iota(E_{ij} v \mbe_k) = \sum_{l \in \Z} ( ld+k + \wt{c}_{k-1} )
E_{(ld+k-1)n + i - 1, (ld+k)n + j - 1}.  \] This can be extended to \begin{equation} \iota(E_{ij} v^s \omega^r \mbe_k) = \sum_{l \in \Z} \left( \prod_{p=0}^{s-1}( ld+k-p + \wt{c}_{k-p-1} ) \right) (-ld-k)^r E_{(ld+k-s)n + i - 1, (ld+k)n + j - 1}. \label{iotav} \end{equation}  The principal grading on the algebra $\ol{M}_{\infty}$ and on the Lie algebra $\ol{\mfgl}_{\infty}$ is given by
$\mathrm{deg}(E_{ij}) = j-i$ and the
embedding $\iota$  respects all the gradings.

We will need, as in \cite{BKLY}, to consider infinite matrices over the ring of truncated
polynomials $R_m = \C[t]/(t^{m+1})$. Fixing $a\in\C$, we define an algebra map
$\varphi_a^{[m]}: M_n(
\msH_{t=1,\mbc}(\Ga) ) \lra \ol{M}_{\infty}(R_m) $ by \[ \varphi_a^{[m]} (E_{ij} u \mbe_k)
= \sum_{l \in \Z} E_{(ld+k+1)n + i - 1, (ld+k)n + j - 1}, \]  \[ \varphi_a^{[m]} (E_{ij}
v \mbe_k) = \sum_{l \in \Z} ( ld+k + a + t + \wt{c}_{k-1}) E_{(ld+k-1)n + i - 1,
(ld+k)n + j - 1}.    \] This extends to a map $ \varphi_a^{[m]} : M_n( \mbA_1 \rtimes \Gamma ) \lra \ol{M}_{\infty}(R_m)$. Explicitly, \[ \varphi_a^{[m]}  (E_{ij} u^{s}
\omega^r \mbe_k) = \sum_{l \in \Z} (-ld-k -a-t)^r E_{(ld+k+s)n + i - 1, (ld+k)n + j - 1}  , \;\; s\in \Z, 0\le k \le d-1.\] 

This embedding when $d=2,n=1$ is related to the embedding considered in \cite{Sh} from the Lie algebra $\mfgl_{\la}$ to $\ol{\mfgl}_{\infty,s}$ (in the notation of \cite{Sh}) since $\mfgl_{\la}$ is obtained by turning into a Lie algebra a certain primitive quotient of $\mfU\mfsl_2(\C)$ and this primitive quotient is isomorphic to the spherical subalgebra of $\msH_{t=1,c=\lambda}(\Z/2\Z)$. 

\subsection{Geometric interpretation}

It was observed in \cite{KaRa} that the algebra of holomorphic differential operators on $\C^{\times}$ has a
geometric interpretation in terms of a certain infinite dimensional vector bundle over
the cylinder $\C / \Z$. The algebras $\mbA_1 \rtimes \Ga$ and $A_1 \rtimes \Ga$ afford
similar interpretations: to explain it, we have to extend them to a holomorphic
setting.

Let $ \mcO(w) $ be the ring of entire functions (holomorphic on all of $\C$) in the
variable $w$.  Let $ \mbA^{\mcO}_1 \rtimes \Ga $ to be the span of the operators of the form
$u^{r}f(w)\ga$ with $f \in \mcO(w),\, r\in\Z$. This span has an algebra structure extending the one
on $\mbA_1 \rtimes \Ga$. Let $ A^{\mcO}_1 \rtimes \Ga $ be the subalgebra of $\mbA^{\mcO}_1
\rtimes \Ga$ consisting of linear combinations of operators of the form $u^{r}f(w)\ga$ and $v^{s}f(w)\ga$ with $r,s\ge 0$ and $f$ holomorphic.

For $k\in\Z$, we define an automorphism $\theta_k$ of $ \ol{M}_{\infty} $ and of
$\ol{\mfgl}_{\infty}$ by $\theta_k(E_{ij}) = E_{i+k,j+k}$. For $0\le k\le d-1$, let $ \ol{M}_{\infty}^k \subset \ol{M}_{\infty}$ be the subspace of matrices such that the $(i,j)$ entry is zero if $j \not\in \cup_{l\in\Z} [ldn+kn,ldn+(k+1)n[$.  The following definition is
adapted from definition 3.4 in \cite{KaRa}.

\begin{definition}\label{monoloop}
A $(n,d)$-monodromic loop is a holomorphic map $\ell: \C \lra
\ol{M}_{\infty}$ such that $\ell(w) = \ell_0(w) + \cdots + \ell_{d-1}(w)$ with $\ell_k(w-d) = \theta_{n}^d \ell_{k}(w)$ and $\ell_k(w) \in \ol{M}_{\infty}^k$ for $0\le k\le d-1$.
\end{definition}

When $d=1$, the following proposition was established in \cite{KaRa}.

\begin{proposition}
The algebra $M_n( \mbA^{\mcO}_1 \rtimes \Ga)$ is isomorphic to the algebra $\mcL_{n,d}$ of
$(n,d)$-monodromic loops.
\end{proposition}

\begin{proof}
We can construct a map $ M_n(\mbA^{\mcO}_1 \rtimes \Ga) \lra \mcL_{n,d}$ by $E \mapsto ( w
\mapsto \varphi_w^{[0]}(E) )$. That the formula $w \mapsto \varphi_w^{[0]}(E)$ defines an
$(n,d)$-monodromic loop follows from the formula for $\varphi_w^{[0]}$. The inverse is
given in the following way. If, given a monodromic loop $\ell$, the loop $\ell_k$ is concentrated along the $(sn+m)^{th}$ ($0\le m\le n-1$) diagonal below the main one (so $sn+m\ge 0$; if it is above, the argument is similar), so that $\ell_k = \sum_{i=1}^{n} \sum_{l\in\Z} f_{i,l,k}(w) E_{(ld+k+s)n+m+i-1,(ld+k)n+i-1} $, then
the preimage of $\ell_k$ is $ \sum_{i=1}^{n-m} E_{m+i,i}  u^s f_{i,0,k}(-\omega-k)\mbe_k + \sum_{i=1}^{m} E_{i,n-m+i}  u^{s+1} f_{n-m+i,0,k}(-\omega-k)\mbe_k $. 
\end{proof}

Since $ M_n( A^{\mcO}_1 \rtimes \Ga ) \into M_n(\mbA^{\mcO}_1 \rtimes \Ga) $, we can identify
$M_n( A^{\mcO}_1 \rtimes \Ga )$ with the algebra of $(n,d)$-monodromic loops $\ell$ such
that, writing $\ell(w) = \sum_{i,j\in\Z} \ell_{i,j}(w) E_{ij}$, we have
that, if $i=l_1n+p_1-1, j=l_2n+p_2-1, 1\le p_1,p_2 \le n$ and $l_1 < l_2$, then $\ell_{i,j}(w)=0$ for $w=p-l_2d$ and $p=0,\ldots,l_2-l_1-1$.

\section{Highest weight representations for matrix Lie algebras over Cherednik algebras of rank one}\label{hwrep}

Inspired by the papers \cite{GGOR,Gu1}, we suggest a notion of
category $\mcO$ for the Lie algebra $\mfsl_n( \msH_{t=1,\mbc}(\Ga) )$ and we study
certain modules in it which we call
quasi-finite highest weight modules.

\begin{definition} Assume that $t\neq 0$. The category $\mcO\big( \mfsl_n(
\msH_{t=1,\mbc}(\Ga) )\big)$ is the category of finitely generated modules $M$ over
$\mfU\mfsl_n(\msH_{t=1,\mbc}(\Ga))$ upon which $\mfsl_n(v\C[v])$ acts locally nilpotently.
\end{definition}

This definition applies also to the $\Ga$-deformed double current algebras $
\msD_{\la,\mbb}^n(\Ga)$ of \cite{Gu3}. One justification for it is that the Schur-Weyl
functor studied in
\cite{Gu1,Gu3} sends modules in the category $\mcO$ of a rational Cherednik algebra for
$\Ga^{\times l} \rtimes S_l$ to a module in $\mcO\big( \msD_{\la,\mbb}^n(\Ga)\big)$ (for
appropriate values of $\la,\mbb$).  It is possible, using induction, to construct analogs
of Verma modules in this category, and one can ask about the classification of
irreducible (integrable) modules in the category $\mcO\big( \mfsl_n( \msH_{t=1,\mbc}(\Ga)
)\big)$. We will not try to answer this
question. Instead, we will study certain modules in these categories by following the
ideas in \cite{BKLY,KaRa}.

Recall the grading on $\mfsl_n(A_1\rtimes \Ga)$ and the
embeddings $\varphi_a^{[m]}: \mfsl_n(A_1\rtimes \Ga) \into \ol{\mfgl}_{\infty}(R_m)$. The
Lie algebra $\ol{\mfgl}_{\infty}$ has an obvious triangular structure compatible with the
grading given by $\mathrm{deg}(E_{ij}) = j-i$ and the embeddings $\varphi_a^{[m]}$ respect the
grading on the source and target spaces.  The following definition comes naturally from
the triangular structure.

\begin{definition}\label{Verma}
Let $\la_{i,k,r} \in \C, \, 1\le i\le n,\, 0\le k\le d-1,\, r\in\Z_{\ge 0}$ and let $\la
\in \mfgl_n( \msH_{1,\mbc}(\Ga) ) [0]^*$ be given by $\la(E_{ii}w^r\mbe_k) =
\la_{i,k,r}$. Extending $\la$ to a one-dimensional representation $\C_{\la}$ of the Lie algebra $ \mfgl_n(
\msH_{1,\mbc}(\Ga) ) [\ge 0]$ ($=\oplus_{k=0}^{\infty} \mfgl_n( \msH_{1,\mbc}(\Ga) )
[k]$) by letting $\mfgl_n( \msH_{1,\mbc}(\Ga) ) [k]$ act trivially if $k>0$, we define
the Verma module $M(\la)$ by \[ M(\la) = \mfU \mfgl_n( \msH_{1,\mbc}(\Ga) )
\ot_{\mfU\mfgl_n( \msH_{1,\mbc}(\Ga) )[\ge 0]} \C_{\la}. \] \end{definition}

The following lemma and definition are quite standard.

\begin{lemma}\label{irr}
The Verma module $M(\la)$ has a unique irreducible quotient which we denote by $L(\la)$.
\end{lemma}

\begin{definition}\label{hwdef}
A $ \mfgl_n( \msH_{1,\mbc}(\Ga) ) $-module is called a highest weight module of highest
weight $\la \in \mfgl_n( \msH_{1,\mbc}(\Ga) ) [0]^*$  if it is generated by a vector $v$ on which $h\in \mfgl_n(
\msH_{1,\mbc}(\Ga) ) [0]$ acts by multiplication by $\la(h)$ and $ \mfgl_n(
\msH_{1,\mbc}(\Ga) ) [k]$ acts trivially if $k\in\Z_{> 0}$.  A vector with this last property is said to
be singular.
\end{definition}

The highest weight vector which generates the Verma module $M(\la)$ will be denoted
$v_{\la}$.

Given a parabolic subalgebra $\mfq$ of $\mfgl_n( \msH_{1,\mbc}(\Ga) )$ with $\mbb$ the
set of its first $nd$ characteristic polynomials, one can define similarly generalized
Verma modules $ M(\mfq,\la)$ by chosing $\la$ such that $\la(h)=0$ for any $h\in \mfgl_n(
\msH_{1,\mbc}(\Ga) ) [0,\mbb]$, since, in this case, $\la$ descends to $\mfq /
[\mfq,\mfq]$: see proposition \ref{der}.

The goal of this section is to study quasifinite irreducible highest weight modules, so
we have to introduce the next definition.

\begin{definition}\label{qfin}
A graded highest weight module $M$ over $\mfgl_n( \msH_{1,\mbc}(\Ga) )$, $M = \oplus_{k \in\Z} M[k]$, is said to be quasifinite if
$\mathrm{dim}_{\C} M[k] < \infty \, \forall\, k\in\Z$.
\end{definition}

In order to obtain below a condition equivalent to the quasifiniteness of $L(\la)$, we
need one more definition, as in \cite{KaRa}.

\begin{definition}\label{hdegd}
The Verma module $M(\la)$ is said to be highly degenerate if there exists a singular
vector $v \in M(\la)[-1]$ such that $v=Av_{\la}$ with $A \in \mfgl_n( \msH_{1,\mbc}(\Ga)
)[-1]$ and $qdet(A)\neq 0$.
\end{definition}

A few words of explanation are in order. The space $\mfgl_n( \msH_{1,\mbc}(\Ga) )[-1]$ is
spanned by $E_{i+1,i} \omega^r \mbe_l $ and by $E_{1n}  \omega^r u \mbe_l $, so the
entries of a matrix $A$ in $ \mfgl_n( \msH_{1,\mbc}(\Ga) )[-1]$ do not necessarily belong
to a commutative ring. By $qdet(A)$, we thus mean the quasi-determinant of $A$ (which, in this case, is, up to a sign, the product of the non-zero entries of $A$).

\begin{proposition}\label{hdegp}
The Verma module $M(\la)$ is highly degenerate if and only if $\la$ vanishes on $\mfgl_n(
\msH_{1,\mbc}(\Ga) )[0,\mbb]$ for some $nd$ monic (so non-zero) polynomials $\mbb =
(b^{i,l}(\omega))_{1\le i\le n}^{0\le l\le d-1}$.
\end{proposition}

\begin{proof}
The same argument as in the proof of proposition 4.1 in \cite{BKLY} applies. The
polynomials $b^{i,l}(\omega)$ are related to the matrix $A$ in the following way. Since
$A \in  \mfgl_n( \msH_{t=1,\mbc}(\Ga) )[-1]$, it can be written as a linear combination of matrices of
the type $E_{i+1,i} b^{i+1,l}(w) \mbe_l $ for $1\le i\le n-1$ and $ E_{1n} b^{1,l}(w) u
\mbe_l $ with $0\le l\le d-1$. It follows from the proof of proposition \ref{der} that $\mfgl_n(
\msH_{1,\mbc}(\Ga)  )[0,\mbb]$ is spanned by $[B,A]$ for all $B \in \mfgl_n(
\msH_{1,\mbc}(\Ga)) [1]$.
\end{proof}

\begin{proposition}\cite{BKLY}
Given $\la \in \mfgl_n( \msH_{1,\mbc}(\Ga) )[0]^*$ as before, the following conditions
are equivalent: \begin{enumerate} \item $M(\la)$ is highly degenerate. \\
\item $L(\la)$ is quasi-finite. \\
\item $L(\la)$ is a quotient of a generalized Verma module $M(\mfq,\la)$ where all the
characteristic polynomials $\mbb = (b^{i,l}(\omega))_{1\le i\le n}^{0\le l\le d-1}$ of $\mfq$ are
non-zero.
\end{enumerate}
\end{proposition}

\begin{proof}
Proposition \ref{hdegp} shows that (1) and (3) are equivalent.  Let us show that, if all
the polynomials $b^{i,l}(\omega)$ are non-zero, then $\mathrm{dim}_{\C} \big( \mfgl_n(
\msH_{1,\mbc}(\Ga) )[k] / \mfq[k] \big)$ is finite, hence $L(\la)$ is quasi-finite. Under
this assumption, it follows from the proof of lemma \ref{nondeg} that $b_k^{i,l}(w)$ are
non-zero for all $k\in\Z_{\le -1}, 1\le i\le n, 0\le l\le d-1$.  Recall that, for $k<0$,
we can write \[ \mfgl_n( \msH_{1,\mbc}(\Ga) )[k] = \sum_{\stackrel{s,l,i,j}{-sn+j-i=k}}
E_{ij} u^s \C[\omega] \mbe_l \text{ and } \mfq[k] = \sum_{\stackrel{s,l,i,j}{-sn+j-i=k}} E_{ij} u^s
\C[\omega] b^{i,l}_k(w)  \mbe_l. \] Our claim now follows from the observation that
$\mathrm{dim}_{\C} \big( \C[w] / (b^{i,l}_k(w)) \big) < \infty$.

Now suppose that $L(\la)$ is quasi-finite. Then $\mathrm{dim}_{\C} L(\la)[-1] < \infty$,
so, if $\wt{M}(\la)$ denotes the unique maximal submodule of $M(\la)$, then
$\wt{M}(\la)[-1] \neq \{ 0 \}$. All the vectors in $\wt{M}(\la)[-1] \neq \{ 0 \}$ are
singular and at least one satisfies the condition in definition \ref{hdegd}. Therefore,
$M(\la)$ is highly degenerate.
\end{proof}

In theorems \ref{irrepAB} and \ref{irrintC}, we stated a criterion in terms of certain
power series for the integrability of the simple quotients of Verma modules for $
\wh{\mfsl}_n(A),\wh{\mfsl}_n(B)$ and $\wh{\mfsl}_n(C)$. We now want to give a similar
criterion for the quasi-finiteness of $L(\la)$.  To achieve this, given $\la \in
\mfgl_n( \msH_{1,\mbc}(\Ga))[0]^*$ as before, set $d_{i,l,r} =  \la(E_{ii}w^r \mbe_l)$
for $1\le i\le n$ and $D_{i,l}(z) = \sum_{r=0}^{\infty} \frac{d_{i,l,r}}{r!} z^r$.
Recall that a quasipolynomial is a linear combination of functions of the form
$p(z)e^{az}$ where $p(z)$ is a polynomial and $a\in\C$.

\begin{theorem}\label{qfinL}
The module $L(\la)$ is quasi-finite if and only if there exist quasipolynomials
$\phi_{i,l}(z), 1\le i\le n, 0\le l\le d-1 $, such that \begin{equation*} D_{1,l}(z) =
\frac{\phi_{1,l}(z)}{1-e^{dz}}   \end{equation*}
 \begin{equation*} D_{i,l}(z) = \frac{\phi_{1,l}(z) + (1-e^{dz}) \phi_{i,l}(z)}{1-e^{dz}}
\mbox{ if } 2\le i\le n.   \end{equation*}
\end{theorem}

\begin{proof}
The proof is similar to the proof of theorem 4.1 in \cite{BKLY}, using the description of
$\mfgl_n( \msH_{1,\mbc}(\Ga)) [0,\mbb]$ given just before proposition \ref{der}. Let us
explain the differences.  Writing $b^{i,l}(\omega)$ as $\omega^{m_{i,l}} +
f_{i,l,m_{i,l}-1}\omega^{m_{i,l}-1} + \cdots + f_{i,l,0}$, we obtain the equations
$\sum_{r=0}^{m_{i,l}} f_{i,l,r} F_{i,l,r+\wt{r}} =0 $ for $1\le i\le n,
\wt{r}=0,1,\ldots$ where $F_{i,l,r} = d_{i,l,r} - d_{i-1,l,r}$ for $2\le i\le n$ and $f_{i,l,m_{i,l}}=1$. To
express $F_{1,l,r}$ in terms of the $d_{i,l,r}$ we write \[ E_{11} uv(\omega+1)^{r} b_{-1}^{1,l}(\omega+1)\mbe_{l+1} - E_{nn}
vu \omega^{r} b_{-1}^{1,l}(\omega)\mbe_{l} \] as
\begin{equation*} \begin{split} - E_{11} (\omega+1)^{r+1} b_{-1}^{1,l}(\omega+1)\mbe_{l+1} +
(\wt{c}_{l}+1) E_{11} (\omega+1)^r b_{-1}^{1,l}(\omega+1)\mbe_{l+1} \\
+ E_{nn} \omega^{r+1} b_{-1}^{1,l}(\omega)\mbe_{l} - (\wt{c}_{l}+1) E_{nn} \omega^{r} b_{-1}^{1,l}(\omega)\mbe_{l} 
\end{split}
\end{equation*}

We thus see that \begin{equation*}  \begin{split} F_{1,l,r}  =  d_{n,l,r+1} - (\tilde{c}_l+1) d_{n,l,r} -
\sum_{j=0}^{r+1} \left( \begin{array}{c} r+1 \\ j  \end{array} \right) d_{1,l+1,j} 
   + (\wt{c}_{l}+1) \sum_{j=0}^{r} \left( \begin{array}{c} r \\ j  \end{array}
\right)  d_{1,l+1,j} . \end{split} \end{equation*}

Setting $F_{i,l}(z) = \sum_{r=0}^{\infty} F_{i,l,r} \frac{z^r}{r!} $ for $1\le i\le n$,
we can conclude as in \cite{BKLY} that $F_{i,l}(z)$ is a quasipolynomial. For $2\le i\le
n$, we can write $F_{i,l}(z) = D_{i,l}(z) - D_{i-1,l}(z)$, and for $i=1$, we have \[
F_{1,l}(z) = D_{n,l}'(z) - (\wt{c}_{l}+1) D_{n,l}(z) -  (e^{z} D_{1,l+1})'(z) +
(\wt{c}_{l}+1) e^{z} D_{1,l+1}(z) .  \] Here, $D_{i,l}'(z)$ is the derivative of
$D_{i,l}(z)$.  This implies that $ D_{i,l}'(z) - D_{i-1,l}'(z)$ is also a
quasi-polynomial (for $2\le i\le n$) and hence so is \[ (D_{1,l}(z) - e^{z} D_{1,l+1}(z))' - (\wt{c}_{l}+1) ( D_{1,l}(z)-e^{z}
D_{1,l+1}(z) ) .  \] Consequently, $e^{z} D_{1,l+1}(z) - D_{1,l}(z)$ is a
quasi-polynomial, and thus so is $(1-e^{dz}) D_{1,l}(z)$.
\end{proof}

It is possible to construct quasi-finite representations of $\mfgl_n(
\msH_{1,\mbc}(\Ga))$ as tensor products of certain modules. This is where the embeddings
$\varphi_a^{[m]}$ come into play. Unfortunately, they are not necessarily irreducible.

First, we need to construct irreducible representations of $ \ol{\mfgl}_{\infty}(R_m) $
using a standard procedure. An element $\la \in \ol{\mfgl}_{\infty}(R_m)[0]^* $ is
determined by $\la_k^{(j)} = \la(E_{kk} t^j), \, k\in\Z, j=0,\ldots,m$, which we call its
labels, following the terminology in \cite{BKLY}. Using induction from the subalgebra of
upper-triangular matrices and its one-dimensional representation determined by such a
$\la$, we construct a Verma module for $ \ol{\mfgl}_{\infty}(R_m) $ and this Verma module
has a unique irreducible highest weight quotient $L(m,\la)$.

\begin{proposition}[\cite{BKLY} proposition 4.4]
The irreducible $ \ol{\mfgl}_{\infty}(R_m) $-module $L(m,\la)$ is quasi-finite if and
only if for each $j=0,\ldots,m$ all but finitely many of the $\la_k^{(j)} -
\la_{k+1}^{(j)}$ are zero.
\end{proposition}

Let $\mbm=(m_1,\ldots,m_N) \in \Z_{\ge 0}^{\oplus N}$ and $\mathbf{\la}=(\la(1),
\ldots, \la(N))$ with $\la(i) \in \ol{\mfgl}_{\infty}(R_{m_i})[0]^* $ such that $L(m_i,\la(i)) $ is quasi-finite. We can form the tensor
product $L(\mbm,\mathbf{\la}) = \bigotimes_{i=1}^N L(m_i,\la(i)) $, which is an
irreducible quasi-finite representation of $\ol{\mfgl}_{\infty}[\mbm] = \oplus_{i=1}^N
\ol{\mfgl}_{\infty}(R_{m_i})$. By pulling it back via the map $\varphi_{\mba}^{[\mbm]} =
\oplus_{i=1}^N \varphi_{a_i}^{[m_i]}: \mfgl_n(\msH_{t=1,\mbc}(\Ga)) \lra
\ol{\mfgl}_{\infty}[\mbm]$ for $\mba=(a_1,\ldots,a_N) \in \C^N$, we obtain a
representation of $\mfgl_n(\msH_{t=1,\mbc}(\Ga))$ which we denote
$\msL_{\mba}(\mbm,\mathbf{\la})$. 

Theorem 4.2 in \cite{BKLY} does not hold for $\mfgl_n(\msH_{t=1,\mbc}(\Ga))$, so we cannot deduce that the representation $\msL_{\mba}(\mbm,\mathbf{\la})$ is necessarily
irreducible. Let us discuss what is the difference here. Theorem 4.2 in \textit{loc.
cit.} states that pulling back a quasifinite representation of $\ol{\mfgl}_{\infty}[\mbm]$ to
$\mfgl_n(\mbA_1 \rtimes \Ga)$ via $\varphi_{\mba}^{[\mbm]}$ gives a representation which
has the same submodules. (The proof in the case $\Ga=\{ 1 \}$ extends to any $d>1$.) The
main ideas of the proof are the following (see also \cite{KaRa}).  First, we have to introduce a holomorphic
enlargement of $\mfgl_n(\mbA_1 \rtimes \Ga)$, as at the end of section \ref{matdiff}: it is the Lie algebra $\mfgl_n(\mbA_1^{\mcO}
\rtimes \Ga)$ spanned by  $E_{ij} u^s f(\omega)\mbe_l$ where $f(\omega)$ is an entire
function of $\omega$, the bracket of $\mfgl_n(\mbA_1 \rtimes \Ga)$ extending to
$\mfgl_n(\mbA_1^{\mcO} \rtimes \Ga)$ naturally. Secondly, the formula for the embedding
$\varphi_{\mba}^{[\mbm]}$ (when $a_i \neq a_j$ for $i\neq j$) can be used to obtain a map
$\varphi_{\mba}^{[\mbm],\mcO}: \mfgl_n(\mbA_1^{\mcO} \rtimes \Ga ) \lra
\ol{\mfgl}_{\infty}[\mbm]$, which is onto, but not necessarily into. The last step is to show that,
if $V$ is a quasi-finite module over $\mfgl_n(\mbA_1 \rtimes \Ga)$, then, by continuity,
we can make $\mfgl_n(\mbA_1^{\mcO} \rtimes \Ga)[k]$ act on $V$ if $k\neq 0$. This
involves computing an upper bound on the norm of certain operators.

The first and third step work also for $\msH_{t=1,\mbc}(\Ga)$, but not the second one.
We consider the algebra $\msH_{t=1,\mbc}^{\mcO}(\Ga)$ which is spanned by elements of
the form $v^r f(\omega)\mbe_l$ and $u^s f(\omega)\mbe_l$ where $f(\omega)$ is an entire
function of $\omega$ and the multiplication is given by (in the case $r\ge s$) \begin{equation*} v^r
f(\omega)\mbe_{l_1} u^s g(\omega)\mbe_{l_2} = \delta_{l_1-s,l_2} v^{r-s} \left(
\prod_{k=1}^{s} ( -\omega + c_{l_2+s-k} + 1+s- k) \right) f(\omega - s)
g(\omega)\mbe_{l_2}  \end{equation*}  We have also a map $\varphi_{\mba}^{[\mbm],\mcO}:
\mfgl_n(\msH_{t=1,\mbc}^{\mcO}(\Ga)) \lra \ol{\mfgl}_{\infty}[\mbm]$, but it is not onto:
for instance, if $a=0=m = c_l$ for $l=0,\ldots, d-1$, then $\varphi_a^{[m]}(E_{ij} v
f(\omega) \mbe_k) = \sum_{l \in \Z} (ld+k) f(-ld-k) E_{(ld+k-1)n + i - 1, (ld+k)n + j -
1}$. Therefore, in the image, the coefficient of $E_{-n + i - 1, j - 1}$ is always zero,
independently of $f(\omega)$. At least, we have the following result.

\begin{proposition}
Assume that $a_i - a_j \not\in\Z$ for $1\le i\neq j\le N$ and $\wt{c}_k + a_i \not\in \Z$ for all $1\le i \le N, \, 0\le k \le d-1$. Then the map $\varphi_{\mba}^{[\mbm],\mcO}:
\mfgl_n( \msH_{t=1,\mbc}^{\mcO}(\Ga) ) \lra \ol{\mfgl}_{\infty}[\mbm]$ is onto.
\end{proposition}

\begin{proof} Decompose $\ol{\mfgl}_{\infty} $ as $\ol{\mfgl}_{\infty} = \ol{\mfn}^{\, -}_{\infty} \oplus \ol{\mfh}_{\infty} \oplus \ol{\mfn}^{\, +}_{\infty}$, where $\ol{\mfh}_{\infty}$ is the Lie subalgebra of all the diagonal blocks of size $n$ (with one having a corner at the $(0,0)$-entry), and $\ol{\mfn}^{\pm}_{\infty}$ are the complements of $\ol{\mfh}_{\infty}$ consisting of strickly upper and lower triangular matrices.  That $\varphi_{\mba}^{[\mbm],\mcO}$ is onto the subspace $\ol{\mfn}^{\, -}_{\infty}$ when restricted to the subspace spanned by the elements $E_{ij}u^s f(\omega) \mbe_k$ with $s\in\Z_{\ge 0}, 1\le i,j \le n, 0\le k\le d-1$ follows from \cite{BKLY,KaRa}, so let us focus instead on $\mfgl_n( \msH_{t=1,\mbc}^{\mcO}(\Ga) )[> 0]$.  Explicitly, using the Taylor formula for the expansion of a funtion of $t$ around $t=0$ and \eqref{iotav}, $ \varphi_{a_i}^{[m_i]} $ is given by \begin{equation} \varphi_{a_i}^{[m_i]} (E_{ij} v^s f(\omega)\mbe_k ) = \sum_{l \in \Z} \sum_{b=0}^{m_i} \frac{g^{(b)}(a_i + ld)}{b!} t^b  E_{(ld+k-s)n + i - 1, (ld+k)n + j - 1} \label{o1} \end{equation} if we set $g(t) = \left( \prod_{p=0}^{s-1}( k-p+t + \wt{c}_{k-p-1} ) \right) f(-k-t)$. As in proposition 3.1 in \cite{KaRa}, we can use the fact that, for every discrete set of points in $\C$, there is a holomorphic function on $\C$ with prescribed values of its first $m_i$ derivatives at each points of such a set. Combining this with our assumption that $a_i - a_j \not\in\Z$ for $1\le i\neq j\le N$, we deduce that, given a matrix $E = \oplus_{i=1}^N E_i$ in $\ol{\mfgl}_{\infty}[\mbm]$, there exists an entire function $g(t)$ such that the right-hand side of \eqref{o1} is equal to $E_i$ for all $i=1,\ldots,N$. To complete the proof, we have to find an entire function $f(\omega)$ such that, if we set $\wt{g}(t) = \left( \prod_{p=0}^{s-1}( k-p + t + \wt{c}_{k-p-1} ) \right) f(-k-t)$, then $g^{(b)}(a_i+ld) = \wt{g}^{(b)}(a_i+ld)$ for $1\le i\le N$, $0\le b\le m_i$ all $l\in\Z$. Set $P(t) = \prod_{p=0}^{s-1}( k-p + t + \wt{c}_{k-p-1} ) $, so that \[ \wt{g}^{(b)}(t) = \sum_{a=0}^b \left( \begin{array}{c} b \\ a \end{array} \right)  P^{(n-a)}(t) f^{(a)}(-k-t).  \] 

Fix $1\le i\le N, \,l\in\Z$ and consider the system of equations \[ g^{(b)}(a_i+ld) = \sum_{a=0}^b \left( \begin{array}{c} b \\ a \end{array} \right)  P^{(b-a)}(a_i+ld) z_{a} \] for $b=0,1,\ldots,m_i$, $z_0,\ldots,z_{m_i}$ being the unknown variables (which we would like to express in terms of $\wt{g}^{(b)}(a_i+ld)$).  Our hypothesis that $\wt{c}_k + a_i \not\in \Z$ implies that $P(a_i+ld) \neq 0$, so the matrix of this system is triangular with non-zero entries along the diagonal. We can thus solve it: let $\wt{z}_{i,l}^0, \wt{z}_{i,l}^1,\ldots,\wt{z}_{i,l}^{m_i}$ be a solution. Then we can rephrase the problem by saying that we now have to find an entire function $f(\omega)$ such that $f^{(a)}(a_i+ld) = \wt{z}_{i,l}^{a}$ for $1\le i\le N$, $0\le a \le m_i$ and all $l\in\Z$. To deduce the existence of such a function, we can now apply the same argument as the one used to deduce the existence of $g(\omega)$ above. \end{proof}

The representation $\msL_{\mba}(\mbm,\mathbf{\la})$ is a highest weight module, so it is
interesting to calculate its associated series $\mbD_{i,k}(z)$, which is equal to
$\sum_{j=1}^N D_{i,j,k}(z)$. The formulas are similar to those in \cite{BKLY}. Set
$h_l^{(p)}(j) = \la_l^{(p)}(j) - \la_{l+1}^{(p)}(j), \, g_{j,l}(z) = \sum_{p=0}^{m_j}
h_l^{(p)}(j) \frac{(-z)^p}{p!}$. We have \begin{equation*} D_{i,j,k}(z)  =  \sum_{p=1}^{m_j} \sum_{l \in \Z}
\la_{(ld+k)n +i-1}^{(p)}(j) \frac{(-z)^p}{p!} e^{-(a_j+ld+k)z} \end{equation*} and \begin{equation*}
\begin{split} D_{i,j,k}(z)  = &
(1-e^{dz})^{-1} \sum_{l\in\Z} e^{-(a_j+ld+k)z} \big( g_{j,(ld+k)n +i-1}(z) + g_{j,(ld+k)n+i}(z) +
\cdots \\ & + g_{j,(ld+k)n + dn + i-2}(z) \big)  \end{split} \end{equation*}

\section{Further discussions}\label{further}

In this section we present further possible research directions related 
to the results of the present paper. 

\subsection{Double affine Lie algebras and Kleinian singularities}

In this section we present further possible research directions related 
to the results of the present paper. $G$ will be an arbitrary finite subgroup of
$SL_2(\C)$. Such a group $G$ does not always act on the torus
$\C^{\times 2}$ or on $\C\times\C^{\times}$, so we can consider only
the algebras $\C[u,v]\rtimes G$ and $\C[u,v]^{G}$. Moreover, when
$G$ is not cyclic, each of these Lie algebras has only one
triangular decomposition, namely: \[ \mfsl_n(\C[u,v]\rtimes G) \cong
\mfn^-(\C[u,v]\rtimes G) \oplus \mfh(\C[u,v]\rtimes G) \oplus
\mfn^+(\C[u,v]\rtimes G) \] \[ \mfsl_n(\C[u,v]^G) \cong
\mfn^-(\C[u,v]^G) \oplus \mfh(\C[u,v]\rtimes G) \oplus
\mfn^+(\C[u,v]^G) \]

These also admit universal central extensions. Since $\C[u,v]^{G}$ is commutative, we have that 
$HC_1(\C[u,v]^{G}) = \frac{\Omega^1(\C[u,v]^{G})}{d(\C[u,v]^{G})}$. We know from
\cite{Ka} that the bracket on the universal central extension $\wh{\mfsl}_n(\C[u,v]^G)
= \mfsl_n(\C[u,v]^G) \oplus \frac{\Omega^1(\C[u,v]^{G})}{d(\C[u,v]^{G})}$ of
$\mfsl_n(\C[u,v]^G)$ is given by \[ [m_1 \ot p_1, m_2 \ot p_2 ] = [m_1,m_2] \ot
(p_1p_2) + Tr(m_1m_2)p_1 dp_2 \]   As for $\wh{\mfsl}_n( \C[u,v]\rtimes G )$, it is
known that its kernel $HC_1(\C[u,v]\rtimes G)$ is equal to
$\frac{\Omega^1(\C[u,v])^{G}}{d(\C[u,v]^{G})}$ (see \cite{Fa}), but to obtain an
explicit formula for its bracket, one would have to choose a splitting $\langle
\C[u,v]\rtimes G  , \C[u,v]\rtimes G \rangle = [ \C[u,v]\rtimes G , \C[u,v]\rtimes G
] \oplus HC_1(\C[u,v]\rtimes G)$ - see section \ref{matrings}.

In \cite{KaVa}, the authors proved that the derived category of coherent sheaves on the
minimal resolution $\wt{\C^2 / G}$  of the singularity $\C^2 / G$ is equivalent to
the derived category of modules over the skew-group ring $\C[u,v] \rtimes G$. It is
thus natural to ask if there is a connection between the derived category of
representations of $\mfsl_n(\C[u,v]\rtimes G)$ and the derived category of modules over
a certain sheaf of Lie algebras on $\wt{\C^2 / G}$.

In the same line of thought, since $A_1 \rtimes G$ and $A_1^{G}$ are Morita
equivalent, one can wonder about the connections between $\mfsl_n(A_1\rtimes G)$ and
$\mfsl_n(A_1^G)$. We observe, however, that even if $A$ and $B$ are Morita
equivalent rings, the categories of representations of $\mfsl_n(A)$ and $\mfsl_n(B)$ are
not necessarily equivalent: as a counterexample, one can consider $A=\C[t^{\pm 1}]$ and
$B = \C[u^{\pm 1}] \rtimes (\Z / d\Z) \cong M_d(\C[t^{\pm 1}])$, in which case
$\mfsl_n(B) = \mfsl_{nd}(\C[t^{\pm 1}])$.

The definitions of Weyl modules from section
\ref{Weylmod} can be adapted to $\mfsl_n(\C[u,v]\rtimes G)$ and
$\mfsl_n(\C[u,v]^G)$. Studying these appear to be a reasonable way to approach the
representation theory of these Lie algebras since, when $G$ is not cyclic, we do not
have triangular decompositions similar to \eqref{td3} or presentations as in proposition
\ref{presC}.  One interesting question is to compute the dimension of local Weyl modules
at the Kleinian singularity. For smooth points on an affine variety and certain highest weights, the dimension of local Weyl modules  has
been computed in \cite{FeLo}, and the case of a double point was treated in \cite{Ku}.
One can expect the study of such local Weyl modules to be related to the geometry of the
minimal resolution of the Kleinian singularity.

\subsection{Quiver Lie algebras}\label{QLiealg}

Symplectic reflection algebras for wreath products of $G$ are known
to be Morita equivalent to certain deformed preprojective algebras of
affine Dynkin quivers which are called Gan-Ginzburg algebras in the
literature \cite{GaGi}. In the rank one case, these are the usual deformed
preprojective algebras $\Pi^{\la}(Q)$. The affine Dynkin diagram in question is
associated to $G$ via the McKay correspondence. The quantum Lie algebra analogs of
these Gan-Ginzburg algebras were introduced in \cite{Gu4} and are
deformations of the enveloping algebra of a Lie algebra which is
slightly larger than the universal central extension of
$\mfsl_n(\Pi(Q))$, where $\Pi(Q)=\Pi^{\la=0}(Q)$. The same themes as in the previous
sections can be studied in the context of the
Lie algebra $\mfsl_n(\Pi(Q))$, in particular when the graph underlying
$Q$ is an affine Dynkin diagram. Actually, when $Q$ is the cyclic quiver on $d$
vertices, $\Pi(Q)\cong \C[u,v]\rtimes \Ga$. Furthermore, if $e_0$ is the extending
vertex of an affine Dynkin diagram, then $e_0 \Pi(Q) e_0 \cong \C[u,v]^{G}$.

All these are examples of matrix Lie algebras over interesting non-commutative rings. It is possible to replace $\mfsl_n$ by another semisimple Lie algebra: this is explained in \cite{BeRe}. It would also be interesting to compare our work with the constructions in \cite{HOT}.

\bigskip

\footnotesize{

\noindent Nicolas Guay \\
Department of Mathematical and Statistical Sciences \\
University of Alberta \\
CAB 632 \\
Edmonton, Alberta  T6G 2G1 \\
Canada
\medskip

\noindent E-mail address: nguay@math.ualberta.ca

\bigskip

\noindent David Hernandez \\ CNRS - \'Ecole Normale Sup\'erieure
 \\
45, rue d'Ulm \\
75005 Paris, \\
France

\medskip

\noindent E-mail address: David.Hernandez@ens.fr

\bigskip

\noindent Sergey Loktev \\
Institute for Theoretical and Experimental Physics \\
Moscow 117218 \\
Russia

\medskip

\noindent E-mail address:  loktev@itep.ru

}


\begin{thebibliography}{99}

\bibitem[AABGP]{AABGP} B. Allison, S. Azam, S. Berman, Y. Gao, A. Pianzola, \emph{Extended affine Lie algebras
and their root systems},  Mem. Amer. Math. Soc.  \textbf{126}  (1997),  no. 603, x+122 pp.

\bibitem[AFLS]{AFLS} J. Alev, M. Farinati, T. Lambre, A. Solotar, \emph{Homologie des
invariants d'une alg\`{e}re de Weyl sous l'action d'un groupe fini},  J. Algebra
\textbf{232}  (2000),  no. 2, 564--577.

\bibitem[BeRe]{BeRe} A. Berenstein, V. Retakh, \emph{Lie algebras and Lie groups over
non-commutative rings}, Advances in Mathematics, \textbf{218} (2008), no. 6, 1723--1758.


\bibitem[BGT]{BGT} S. Berman, Y. Gao, S. Tan, \emph{A unified view of some vertex operator constructions}, Israel J. Math.  \textbf{134}, no. 1, 29--60.

\bibitem[BKLY]{BKLY} C. Boyallian, V. Kac, J. Liberati, C. Yan, \emph{Quasfinite highest
weight modules over the Lie algebra of matrix differential operators on the circle}, J.
of Math. Phys. \textbf{39} (1998), no. 5, 2910--2928.

\bibitem[BoLi]{BoLi} C. Boyallian, J. I. Liberati, \emph{On modules over matrix quantum
pseudo-differential operators},
Lett. Math. Phys. \textbf{60} (2002), no. 1, 73--85.

\bibitem[C]{ca} V. Chari, \emph{Integrable representations of affine
Lie-algebras}, {Invent. Math. {\bf 85}  (1986),  no. 2, 317--335}

\bibitem[CB]{CB} W. Crawley-Boevey, \emph{Regular modules for tame
hereditary algebras}, Proc. London Math. Soc., \textbf{62} (1991),
490--508.

\bibitem[ChLe]{cl} V. Chari and T. Le, \emph{Representations of double affine Lie
algebras}, A tribute to C. S. Seshadri (Chennai, 2002),  199--219,
Trends Math., Birkhauser, Basel, 2003.

\bibitem[ChLo]{ChLo} V. Chari, S. Loktev, \emph{Weyl, Fusion and Demazure modules for the
current algebra of $\mfsl_{r+1}$}, Adv. Math. \textbf{207} (2006), 928--960.

\bibitem[ChPr]{ChPr} V. Chari, A. Pressley, \emph{Weyl Modules for Classical and Quantum
Affine algebras},  Represent. Theory  \textbf{5}  (2001), 191--223.

\bibitem[EG]{EtGi} P. Etingof, V. Ginzburg,  \emph{Symplectic reflection
algebras, Calogero-Moser space, and deformed Harish-Chandra homomorphism},
     Invent. Math.  \textbf{147}  (2002),  no. 2, 243--348.

\bibitem[Fa]{Fa} M. Farinati, \emph{Hochschild duality, localization,
and smash products}, J. Algebra \textbf{284} (2005), no. 1, 415--434.

\bibitem[FeLo]{FeLo} B. Feigin, S. Loktev, \emph{Multi-dimensional Weyl
modules and symmetric functions}, Comm. Math. Phys. \textbf{251}
(2004), no. 3, 427--445.

\bibitem[FoLi]{FoLi} G. Fourier, P. Littelmann, \emph{Weyl modules, Demazure modules,
KR-modules, crystals, fusion products and limit constructions},
Adv. Math. \textbf{211} (2007), no. 2, 566--593.

\bibitem[Ga]{Ga} H. Garland, \emph{The arithmetic theory of loop algebras}, J. Algebra, \textbf{53} (1978), 480--551.

\bibitem[GaGi]{GaGi} W.L. Gan, V. Ginzburg, \emph{Deformed
 preprojective algebras and symplectic reflection algebras for wreath products},  J.
Algebra  \textbf{283}  (2005),  no. 1, 350--363.

\bibitem[GGOR]{GGOR} V. Ginzburg, N. Guay, E. Opdam, R. Rouquier, \emph{On the category
$\mathcal{O}$ for rational Cherednik algebras},
Invent. Math.  \textbf{154}  (2003),  no. 3, 617--651.

\bibitem[GKV]{gkv} V. Ginzburg, M. Kapranov and E. Vasserot, {\it
Langlands reciprocity for algebraic surfaces}, Math. Res. Lett. {\bf
2} (1995), no. 2, 147--160

\bibitem[Go1]{Go1} I. Gordon, \emph{On the quotient ring by diagonal invariants}, 
Invent. Math. \textbf{153} (2003), no. 3, 503--518.

\bibitem[Go2]{Go2} I. Gordon, \emph{Gelfand-Kirillov conjecture for
symplectic reflection algebras}, arXiv:0710.14.19.

\bibitem[Gu1]{Gu1} N. Guay, \emph{Cherednik algebras and Yangians},  Int.
Math. Res. Not. \textbf{2005}, no.57, 3551--3593.

\bibitem[Gu2]{Gu2} N. Guay, \emph{Affine Yangians and deformed double
current algebras in type $A$},  Adv. Math.  \textbf{211}  (2007),  no.
2, 436--484.

\bibitem[Gu3]{Gu3} N. Guay, \emph{Quantum algebras and symplectic
reflection algebras for wreath products}, submitted for publication.

\bibitem[Gu4]{Gu4} N. Guay, \emph{Quantum algebras and quivers}, Selecta Math. (N.S.) \textbf{14} (2009), no. 3-4, 667--700.

\bibitem[Ha1]{Ha1} M. Haiman, \emph{Conjectures on the quotient ring by diagonal invariants},  J. Algebraic Combin.  \textbf{3}  (1994),  no. 1, 17--76.

\bibitem[Ha2]{Ha2} M. Haiman, \emph{Vanishing theorems and character formulas for the Hilbert scheme of points in the plane}, 
Invent. Math. \textbf{149} (2002), no. 2, 371--407.

\bibitem[H1]{He1} D. Hernandez, \emph{Representations of quantum
affinizations and fusion product}, Transform. Groups \textbf{10} (2005),
no. 2, 163--200.

\bibitem[H2]{He2} D. Hernandez, \emph{Drinfeld coproduct, quantum
fusion tensor category and applications}, Proc. London Math. Soc. (3)
{\bf 95} (2007),  no. 3, 567--608.

\bibitem[H3]{He3} D. Hernandez, \emph{Quantum toroidal algebras and
their representations}, Selecta Math. (N.S.) \textbf{14} (2009), no. 3-4, 701--725.

\bibitem[HOT]{HOT} G. Halbout, J-M. Oudom, X. Tang, \emph{Deformations of linear Poisson orbifolds}, arXiv:0807.0027.

\bibitem[K1]{K} V. Kac, \emph{Simple irreducible graded Lie
algebras of finite growth},  Izv. Akad. Nauk SSSR Ser. Mat.
\textbf{32}  1968 1323--1367.

\bibitem[K2]{kac} V. Kac, \emph{Infinite dimensional Lie algebras},
3rd Edition, Cambridge University Press (1990).

\bibitem[Ka]{Ka} C. Kassel, \emph{K\"ahler differentials and coverings
of complex simple Lie algebras extended over a commutative algebra},
Proceedings of the Luminy conference on algebraic $K$-theory (Luminy,
1983),
J. Pure Appl. Algebra \textbf{34} (1984), no. 2-3, 265--275.

\bibitem[KaRa]{KaRa} V. Kac, A. Radul, \emph{Quasifinite highest weight modules over the
Lie algebra of differential operators on the circle}, Commun. Math. Phys. \textbf{157}
(1993), 429--457.

\bibitem[KaVa]{KaVa} M. Kapranov, E. Vasserot, \emph{Kleinian singularities, derived
categories and Hall algebras},  Math. Ann.  \textbf{316}  (2000),  no. 3, 565--576.

\bibitem[Kh]{Kh} A. Khare, \emph{Category $\mcO$ over skew group rings}, to appear in
Communications in Algebra.

\bibitem[KL]{KL} C. Kassel, J.L. Loday,  \emph{Extensions centrales
d'alg\`ebres de Lie}, Ann. Inst. Fourier (Grenoble) \textbf{32} (1982),
no. 4, 119--142 (1983).

\bibitem[Ku]{Ku} T. Kuwabara, \emph{Symmetric coinvariant algebras and local Weyl modules
at a double point}, J. Algebra \textbf{295} (2006), no. 2, 426--440.

\bibitem[Lo]{Lo} S. Loktev, \emph{Weight multiplicity polynomials of multi-variable Weyl modules}, arXiv: 0806.0170.

\bibitem[Ma1]{Ma1} O.Mathieu, \emph{Classification des alg\`ebres de
Lie gradu\'ees simples de croissance $\le 1$}, Invent. math.
\textbf{86} (1986), no. 2, 371-426.

\bibitem[Ma2]{Ma2} O.Mathieu, \emph{Classification of simple graded
Lie algebras of finite growth}, Invent. math. \textbf{108} (1992), no.
, 455-519.

\bibitem[Mi1]{Mi1} K. Miki, \emph{Representations of quantum toroidal algebra
$U\sb q({\rm sl}\sb {n+1,{\rm tor}}) (n\geq 2)$}, J. Math. Phys. {\bf
41} (2000), no. 10, 7079--7098.

\bibitem[Mi2]{Mi2} K. Miki, \emph{Integrable irreducible highest weight 
modules for ${\rm sl}\sb 2(\mathbf C\sb p[x\sp {±1},y\sp {±1}])$}, 
Osaka J. Math. \textbf{41} (2004), no. 2, 295--326.

\bibitem[MRY]{MRY} R.V. Moody, S. E. Rao, T. Yokonuma, \emph{Toroidal
Lie algebras and vertex representations}, Geom. Dedicata  \textbf{35}
(1990),  no. 1-3, 283--307.

\bibitem[Nag]{nag} K. Nagao, \emph{K-theory of quiver varieties, q-Fock space and nonsymmetric Macdonald polynomials},
Preprint arXiv:0709.1767.

\bibitem[Nak1]{Naams} H. Nakajima, \emph{Quiver varieties and
finite-dimensional representations of quantum affine algebras}, J.
Amer. Math. Soc. {\bf 14} (2001), no. 1, 145--238.

\bibitem[Nak2]{npc} H. Nakajima, \emph{Geometric construction of
representations of affine algebras}, Proceedings of the International
Congress of Mathematicians, Vol. I (Beijing, 2002), 423--438, Higher
Ed. Press, Beijing (2002).

\bibitem[Ra]{Ra1} S. E. Rao, \emph{On representations of toroidal Lie algebras}
Functional analysis VIII,  146--167, Various Publ. Ser. (Aarhus), 47, Aarhus Univ.,
Aarhus, 2004.

\bibitem[S]{Sc2} O. Schiffmann, \emph{Canonical bases and moduli
spaces of sheaves on curves}, Invent. Math. {\bf 165} (2006), no. 3,
453--524.

\bibitem[Sh]{Sh} B. Shoiket, \emph{Certain topics on the representation theory Lie algebra $gl(lambda)$}, 	Journal of Mathematical Sciences \textbf{92} (1998), no.2, 	3764--3806.

\bibitem[Va]{Va} R. Vale, \emph{Rational Cherednik algebras and diagonal coinvariants of $G(m,p,n)$}, 
J. Algebra \textbf{311} (2007), no. 1, 231--250.


\bibitem[VV1]{VaVa1} M. Varagnolo, E. Vasserot, \emph{Schur duality
in the toroidal setting},  Comm. Math. Phys.  \textbf{182}  (1996),
no. 2,
469--483.

\bibitem[VV2]{VaVa2} M. Varagnolo, E. Vasserot, \emph{Double-loop
algebras and the Fock space}, Invent. Math. \textbf{133} (1998), no. 1,
133--159.

\bibitem[Y]{yy} Y. Yoon, \emph{On the polynomial representations of current algebras}, J.
Algebra {\bf 252} (2002), no. 2, 376--393.

\end{thebibliography}
\end{document}